\documentclass[10pt,a4paper,pointlessnumbers]{scrartcl}
\usepackage[T1]{fontenc}

\usepackage{amssymb}
\usepackage{amsmath}
\usepackage{amsfonts}
\usepackage{bbm}
\usepackage{stmaryrd}
\usepackage{amsthm}
\usepackage{lmodern}
\usepackage{mathrsfs}
\usepackage{hyperref}
\hypersetup{linktocpage,colorlinks=true,citecolor=purple,linkcolor=blue,urlcolor=blue,filecolor=black}

\usepackage[margin=2.3cm]{geometry}
\usepackage[all,cmtip]{xy}
\usepackage[utf8]{inputenc}
\usepackage{graphicx}

\usepackage{varwidth}
\usepackage{overpic}

\makeatletter
\DeclareRobustCommand{\em}{%
	\@nomath\em \if b\expandafter\@car\f@series\@nil
	\normalfont \else \slshape \fi}
\makeatother
\usepackage{upgreek}	
\usepackage{rotating}
\usepackage{tikz}
\usetikzlibrary{matrix,arrows,decorations.pathmorphing,shapes.geometric}
\usepackage{tikz-cd}
\usepackage{needspace}

\usetikzlibrary{decorations.markings}

\usetikzlibrary{backgrounds}

\newcommand{\cA}{\mathcal{A}}
\newcommand{\Ao}{\cA_{\text{\normalfont \bfseries !}}}

\newcommand{\cB}{\cat{B}}
\newcommand{\cC}{\cat{C}}

\newcommand{\cE}{\cat{E}}
\newcommand{\cM}{\cat{M}}
\newcommand{\cN}{\cat{N}}

\newcommand{\N}{\mathsf{N}}
\newcommand{\close}{\catf{cl}}
\newcommand{\sS}{\mathsf{S}}

\newcommand{\HOM}{\underline{\catf{Hom}}} \newcommand{\END}{\underline{\catf{\End}}}

\usepackage{tikz-cd}
\usepackage{mathtools}
\usetikzlibrary{decorations.markings}
\usetikzlibrary{backgrounds}

\usepackage{etex}

\tikzstyle{tikzfig}=[baseline=-0.25em,scale=0.5]

\pgfkeys{/tikz/tikzit fill/.initial=0}
\pgfkeys{/tikz/tikzit draw/.initial=0}
\pgfkeys{/tikz/tikzit shape/.initial=0}
\pgfkeys{/tikz/tikzit category/.initial=0}

\pgfdeclarelayer{edgelayer}
\pgfdeclarelayer{nodelayer}
\pgfsetlayers{background,edgelayer,nodelayer,main}

\tikzstyle{none}=[inner sep=0mm]

\newcommand{\tikzfig}[1]{%
	{\tikzstyle{every picture}=[tikzfig]
		\IfFileExists{#1.tikz}
		{\input{#1.tikz}}
		{%
			\IfFileExists{./figures/#1.tikz}
			{\input{./figures/#1.tikz}}
			{\tikz[baseline=-0.5em]{\node[draw=red,font=\color{red},fill=red!10!white] {\textit{#1}};}}%
	}}%
}

\tikzstyle{every loop}=[]

\usepackage{tikzit}

\tikzstyle{black dot}=[fill=black, draw=black, shape=circle, minimum size=3pt, inner sep=0pt]
\tikzstyle{black dot small}=[fill=black, draw=black, shape=circle, minimum size=2pt, inner sep=0pt]
\tikzstyle{fblack dot}=[fill=black, draw=red, shape=circle, minimum size=2pt, inner sep=0pt]
\tikzstyle{wbox}=[fill=white, draw=black, shape=rectangle, minimum height=0.5cm, minimum width=0.01cm]
\tikzstyle{bbox}=[fill=white, draw=blue, shape=rectangle, minimum height=0.5cm, minimum width=0.01cm]
\tikzstyle{rbox}=[fill=white, draw=red, shape=rectangle, minimum height=0.5cm, minimum width=0.01cm]
\tikzstyle{bwbox}=[draw=blue, shape=rectangle, minimum width=2cm, minimum height=0.5cm]
\tikzstyle{bbwbox}=[draw=blue, shape=rectangle, minimum width=1cm, minimum height=1cm]
\tikzstyle{big white circle}=[fill=white, draw=black, shape=circle, minimum width=0.75cm]
\tikzstyle{white dot big}=[fill=white, draw=black, shape=circle, inner sep=1pt]
\tikzstyle{white dot}=[fill=white, draw=black, shape=circle, minimum size=3pt, inner sep=0pt]
\tikzstyle{flat box}=[fill=white, draw=black, shape=rectangle, minimum width=1.3cm, minimum height=0.5cm,fill=morphismcolor]
\tikzstyle{square}=[fill=white, draw=black, shape=rectangle]
\tikzstyle{flat box 2}=[fill=white, draw=black, shape=rectangle, minimum height=0.5cm, minimum width=0.01cm,fill=morphismcolor]
\tikzstyle{bigbox}=[fill=white, draw=black, shape=rectangle, minimum height=0.5cm, minimum width=0.8cm,fill=white]
\tikzstyle{over }=[front]
\tikzstyle{theta}=[fill=blue, draw=blue, shape=ellipse, minimum height=6pt, minimum width=6pt, inner sep=0pt]
\tikzstyle{thetabig}=[fill=blue, draw=blue, shape=ellipse, minimum width=1cm, minimum height=0.01cm]
\tikzstyle{thetainv}=[fill=blue, draw=red, shape=ellipse, minimum height=6pt, minimum width=6pt, inner sep=0pt]
\tikzstyle{thetabinv}=[fill=blue, draw=red, shape=ellipse, minimum width=1cm, minimum height=0.01cm]
\tikzstyle{bigdisk}=[draw=black, shape=circle, minimum width=3cm]
\tikzstyle{wdisk}=[shape=circle, minimum width=0.48cm,fill=white]
\tikzstyle{bigdisk2}=[draw=black, fill=lightgray, shape=circle, minimum width=3cm]
\tikzstyle{little disk}=[fill=white, draw=black, shape=circle, minimum width=0.5cm]

\tikzstyle{mid arrow}=[-, postaction={on each segment={mid arrow}}]
\tikzstyle{end arrow}=[->]
\tikzstyle{mover}=[-, link]
\tikzstyle{string}=[-, draw=blue,postaction={on each segment={mid arrow}}]
\tikzstyle{stringd}=[-, dotted,draw=blue,postaction={on each segment={mid arrow}}]
\tikzstyle{red}=[-, dotted,draw=red]
\tikzstyle{mydots}=[-,dotted,dashed,draw=gray]
\tikzstyle{mydotsblack}=[-,dotted,dashed,draw=black]
\tikzstyle{open}=[-, line width=2pt,draw=blue]
\tikzstyle{thick}=[-,line width=1pt]
\tikzstyle{rarrow}=[->,draw=red]
\tikzstyle{red mid arrow}=[-, draw={rgb,255: red,214; green,42; blue,51}, postaction={on each segment={mid arrow}}, line width=1pt]
\tikzstyle{RED}=[-, draw={rgb,255: red,214; green,42; blue,51}]
\tikzstyle{REDdashed}=[-,dashed, draw={rgb,255: red,214; green,42; blue,51}]
\tikzstyle{REDarrow}=[->, draw={rgb,255: red,214; green,42; blue,51}]
\tikzstyle{darrow}=[->,dotted]
\tikzstyle{blue}=[-, draw=blue]
\tikzstyle{blue mid arrow}=[-, draw={rgb,255: red,23; green,37; blue,167}, postaction={on each segment={mid arrow}}, line width=1pt]
\tikzstyle{over}=[-, link]
\tikzstyle{bover}=[-, blink]
\tikzstyle{mover}=[-, link]
\tikzstyle{mapsto}=[{|->}]

\usetikzlibrary{decorations.pathreplacing,decorations.markings}
\tikzset{
	on each segment/.style={
		decorate,
		decoration={
			show path construction,
			moveto code={},
			lineto code={
				\path [#1]
				(\tikzinputsegmentfirst) -- (\tikzinputsegmentlast);
			},
			curveto code={
				\path [#1] (\tikzinputsegmentfirst)
				.. controls
				(\tikzinputsegmentsupporta) and (\tikzinputsegmentsupportb)
				..
				(\tikzinputsegmentlast);
			},
			closepath code={
				\path [#1]
				(\tikzinputsegmentfirst) -- (\tikzinputsegmentlast);
			},
		},
	},
	mid arrow/.style={postaction={decorate,decoration={
				markings,
				mark=at position .7 with {\arrow[#1]{stealth}}
	}}},
}
\tikzset{%
	link/.style    = { white, double = black, line width = 1.8pt,
		double distance = 0.4pt },
	channel/.style = { white, double = black, line width = 0.8pt,
		double distance = 0.8pt },
}

\tikzset{%
	blink/.style    = { white, double = blue, line width = 2pt,
		double distance = 1pt },
	channel/.style = { white, double = blue, line width = 2pt,
		double distance = 1pt },
}

\tikzstyle{tikzfig}=[baseline=-0.25em,scale=0.5]

\pgfkeys{/tikz/tikzit fill/.initial=0}
\pgfkeys{/tikz/tikzit draw/.initial=0}
\pgfkeys{/tikz/tikzit shape/.initial=0}
\pgfkeys{/tikz/tikzit category/.initial=0}

\pgfdeclarelayer{edgelayer}
\pgfdeclarelayer{nodelayer}
\pgfsetlayers{background,edgelayer,nodelayer,main}

\tikzstyle{none}=[inner sep=0mm]

\tikzstyle{every loop}=[]

\usepackage{mathtools}
\usepackage{epstopdf}

\usepackage[capitalize]{cleveref}

\newtheoremstyle{mytheorem}
{\topsep}
{\topsep}
{\slshape}
{0pt}
{\bfseries}
{.}
{ }
{\thmname{#1}\thmnumber{ #2}\thmnote{ {\normalfont\slshape(#3)}}}

\newtheoremstyle{mydefinition}
{\topsep}
{\topsep}
{\normalfont}
{0pt}
{\bfseries}
{.}
{ }
{\thmname{#1}\thmnumber{ #2}\thmnote{ {\normalfont\slshape(#3)}}}

\theoremstyle{mytheorem}
\newtheorem{theorem}{Theorem}[section]

\makeatletter
\newtheorem*{rep@theorem}{\rep@title}
\newcommand{\newreptheorem}[2]{%
	\newenvironment{rep#1}[1]{%
		\def\rep@title{#2 \ref{##1}}%
		\begin{rep@theorem}}%
		{\end{rep@theorem}}}
\makeatother

\newreptheorem{theorem}{Theorem}
\newreptheorem{corollary}{Corollary}
\newtheorem{lemma}[theorem]{Lemma}
\newtheorem{proposition}[theorem]{Proposition}
\newtheorem{corollary}[theorem]{Corollary}
\theoremstyle{mydefinition}
\newtheorem{definition}[theorem]{Definition}
\newtheorem{exx}[theorem]{Example}
\newenvironment{example}
{\pushQED{\qed}\exx}
{\popQED\endexx}
\newtheorem{remm}[theorem]{Remark}
\newenvironment{remark}
{\pushQED{\qed}\remm}
{\popQED\endremm}
\numberwithin{equation}{section}
\usepackage{enumitem}

\newenvironment{pnum}{\begin{enumerate}[topsep=2pt,parsep=2pt,partopsep=2pt,itemsep=0pt,label={(\roman{*})}]}{\end{enumerate}}

\DeclareMathSymbol{\Phiit}{\mathalpha}{letters}{"08}\let\Phi\undefined\newcommand{\Phi}{\Phiit}
\DeclareMathSymbol{\Psiit}{\mathalpha}{letters}{"09}\let\Psi\undefined\newcommand{\Psi}{\Psiit}
\DeclareMathSymbol{\Sigmait}{\mathalpha}{letters}{"06}\let\Sigma\undefined\newcommand{\Sigma}{\Sigmait}
\DeclareMathSymbol{\Xiit}{\mathalpha}{letters}{"04}
\DeclareSymbolFont{extraup}{U}{zavm}{m}{n}
\DeclareMathSymbol{\vardiamond}{\mathalpha}{extraup}{87}
\DeclareMathSymbol{\Lambdait}{\mathalpha}{letters}{"03}\let\Lambda\undefined\newcommand{\Lambda}{\Lambdait}
\DeclareMathSymbol{\Piit}{\mathalpha}{letters}{"05}\let\Pi\undefined\newcommand{\Pi}{\Piit}
\DeclareMathSymbol{\Gammait}{\mathalpha}{letters}{"00}\let\Gamma\undefined\newcommand{\Gamma}{\Gammait}
\DeclareMathSymbol{\Omegait}{\mathalpha}{letters}{"0A}\let\Omega\undefined\newcommand{\Omega}{\Omegait}
\DeclareMathSymbol{\Upsilonit}{\mathalpha}{letters}{"07}\let\Upsilon\undefined\newcommand{\Upsilon}{\Upilonit}
\DeclareMathSymbol{\Thetait}{\mathalpha}{letters}{"02}\let\Theta\undefined\newcommand{\Theta}{\Thetait}

\def\Hom{\catf{Hom}}
\let\O\undefined\newcommand{\O}{\catf{O}}

\def\End{\catf{End}}
\def\id{\mathrm{id}}

\let\to\undefined\newcommand{\to}{\longrightarrow}
\let\mapsto\undefined\newcommand{\mapsto}{\longmapsto}
\newcommand{\catf}[1]{\mathsf{#1}}
\newcommand{\Proj}{\operatorname{\catf{Proj}}}

\newcommand{\Map}{\catf{Map}}

\def\op{\mathsf{op}}
\def\rev{{\otimes \mathsf{op}}}
\def\env{\mathsf{env}}

\def\tr{\catf{tr}\,}

\newcommand{\ra}[1]{\xrightarrow{\   #1    \ }}

\newcommand{\Lexf}{\catf{Lex}^\catf{f}}
\newcommand{\Rexf}{\catf{Rex}^\catf{f}}
\newcommand{\Rex}{\catf{Rex}}

\newcommand{\naka}{\catf{N}^\catf{r}}

\newcommand{\act}{\triangleright}

\newcommand{\hocolimsub}[1]{\underset{#1}{\operatorname{hocolim}}\,}

\newcommand{\Lex}{\catf{Lex}}

\newcommand{\vect}{\catf{vect}}

\newcommand{\cat}[1]{\mathcal{#1}}

\newcommand{\nakar}{\catf{N}^r}

\usepackage{dingbat}

\makeatletter
\newcommand{\monthyeardate}{%
	\DTMenglishmonthname{\@dtm@month}, \@dtm@year
}
\makeatother

\setkomafont{caption}{\small\slshape}

\setkomafont{section}{\rmfamily\large}
\setkomafont{subsection}{\normalfont\slshape}
\setkomafont{subsubsection}{\rmfamily}
\setkomafont{subparagraph}{\normalfont\scshape}

\newtheorem*{theorem*}{Theorem}
\newtheorem*{corollary*}{Corollary}

\makeatletter
\renewcommand\section{\@startsection {section}{1}{\z@}%
	{-3.5ex \@plus -1ex \@minus -.2ex}%
	{2.3ex \@plus.2ex}%
	{\normalfont\scshape\centering}}
\makeatother

\usepackage{titlesec}
\titleformat{\subsection}[runin]
{\normalfont\slshape}
{\thesubsection}
{0.5em}
{}
[.]

\usepackage{titletoc}
\dottedcontents{section}[1.5em]{\rmfamily}{1.5em}{0.2cm}
\dottedcontents{subsection}[5em]{\rmfamily}{2.3em}{0.2cm}

\begin{document}

	\vspace*{-1.5cm}	\begin{center}	\textbf{\large{Pivotal Module Categories, Factorization Homology \\[0.5ex]
				and Modular Invariant Modified Traces	}} \\ 
		\vspace{1cm}{\large Jorge Becerra \quad and \quad Lukas Woike }\\ 	\vspace{5mm}{\slshape  Université Bourgogne Europe\\ CNRS\\ IMB UMR 5584\\ F-21000 Dijon\\ France }\end{center}	\vspace{0.3cm}	
	\begin{abstract}\noindent 
		The algebraic notion of a pivotal module category was developed by Schaumann and Shimizu and is central to the description of boundary conditions in conformal field theory according to a proposal by Fuchs and Schweigert. In this paper, we present a large class of examples of pivotal module categories of topological origin: For a unimodular finite ribbon category $\mathcal{A}$, we prove that the factorization homology $\int_\Sigma \mathcal{A}$ of a compact oriented surface $\Sigma$ with $n$ marked boundary intervals, at least one per connected component, comes with the structure of a pivotal module category over $\mathcal{A}^{\boxtimes n}$. This endows the internal skein algebras of Ben-Zvi-Brochier-Jordan, in particular the elliptic double, with a symmetric Frobenius structure. As application, we obtain, for each choice of $\mathcal{A}$, a family of full open conformal field theories, each of which comes with correlation functions for all surfaces with marked boundary intervals that are explicitly computable using factorization homology.	As a further application, we explain how modified traces can be `integrated' over surfaces: We show that the modified trace for $\mathcal{A}$ extends in a canonical way to the factorization homology of $\Sigma$. The resulting traces have the remarkable property of being modular invariant, i.e.\ fixed by the mapping class group action.
\end{abstract}

	\tableofcontents
	\normalsize

	\section{Introduction and summary}
	Suppose that we are given a \emph{finite tensor category} $\cA$
	 in the sense of Etingof-Ostrik~\cite{etingofostrik}, i.e.\ a linear, abelian, rigid, monoidal category subject to appropriate finiteness conditions. Such a category could arise for instance 
	from Hopf algebras or within the framework of two-dimensional conformal field theory, see e.g.~\cite{kassel,egno,algcften} for more background.
	While finite tensor categories are already
	 algebraic objects of independent interest that are often of representation-theoretic origin, 
	  they also have a representation theory themselves: They may act on other categories. This is an idea formalized through the notion of a \emph{module category}, see e.g.~\cite[Chapter~7]{egno} for a textbook reference.
	A finite module category $\cM$ over a finite tensor category $\cA$ comes with an action $-\ogreaterthan-:\cA\boxtimes \cM\to \cM$ inside a suitable symmetric monoidal bicategory of finite categories.
	It has an \emph{internal hom} $\HOM:\cM^\op \boxtimes \cM\to \cA$ that is characterized by isomorphisms
$$	\Hom_{\cM}(X \ogreaterthan M,N) \cong \Hom_\cA (X, \HOM(M,N)) $$ natural in $X\in\cA$ and $M,N\in\cM$. 

Let $\cM$ be a finite module category over a finite tensor category $\cA$ whose  rigid duality we denote by $-^\vee$.
We will be interested throughout in the situation in which $\cA$ is \emph{pivotal}, i.e.\ equipped with a monoidal isomorphism $\omega: \id_\cA \cong -^{\vee \vee}$. This identifies the left and right duality, so we will not distinguish between them.
It is now a natural question whether
$\cM$ comes with
natural isomorphisms
$$	\HOM(M,N)^\vee \cong \HOM(N,M) $$
 that square to the identity relative to the pivotal structure. 
The core idea of such a structure goes back to Schaumann~\cite{gstrace,schaumannpiv}; it is defined also by Shimizu~\cite{relserre} under the name \emph{pivotal module category}.

The internal endomorphism algebras of a pivotal module category over a pivotal finite tensor category $\cA$ happen to be symmetric Frobenius algebras in $\cA$~\cite[Theorem~3.15]{relserre}. Relying, among other things, on this extremely important algebraic fact, Fuchs-Schweigert postulate in \cite{fspivotal} that a pivotal module category is exactly the algebraic structure that one should encounter on the category of boundary conditions in logarithmic conformal field theory. More precisely, they explain how the composition of internal endomorphisms of a pivotal module category is the correct framework to describe the boundary operator product expansion.

\subsection*{Pivotal module categories from factorization homology}
The motivation for the notion of a pivotal module category as well as the sources of examples known to date are algebraic.
The purpose of this article is to
exhibit a \emph{topological}  source of pivotal module categories and to use this to rigorously construct families of full open conformal field theories and modular invariant modified traces.

To this end, we will need more structure on the finite tensor category $\cA$. More specifically, we will assume that $\cA$ is a \emph{finite ribbon category}, i.e.\ it is additionally equipped with a braiding and a ribbon structure. 
In particular, $\cA$ is a framed $E_2$-algebra in a suitable symmetric monoidal bicategory of finite categories~\cite{salvatorewahl} which allows to define the \emph{factorization homology} $\int_\Sigma \cA$ of $\cA$ on a compact oriented surface $\Sigma$ (possibly with boundary)
following \cite{AF,higheralgebra}. This is again a finite category that one should think of as the result of `integrating' $\cA$ over $\Sigma$. 

For the datum of $n\ge 1$ parametrized intervals in $\Sigma$, at least one in each connected component of $\Sigma$ (Figure~\ref{figsigma}), the category $\int_\Sigma \cA$ becomes a module over $\cA^{\boxtimes n}$. This $\cA^{\boxtimes n}$-module is considered by Ben-Zvi-Brochier-Jordan in \cite{bzbj} and is used to describe the entire category $\int_\Sigma \cA$ via (generalizations of) the \emph{moduli algebras} of Alekseev-Grosse-Schomerus~\cite{alekseevmoduli,agsmoduli,asmoduli} that appear as the 
internal endomorphism algebras 
\begin{equation}\label{eqnmodulialgebraxy}
	\mathfrak{a}_\Sigma:= \END(\cat{O}_\Sigma^\cA)
\end{equation}	of the \emph{quantum structure sheaf} $\cat{O}_\Sigma^\cA\in\int_\Sigma \cA$, i.e.\ the object in factorization homology induced by the unique embedding $\emptyset \to \Sigma$.
This moduli algebra carries a mapping class group action by algebra maps constructed in the semisimple Hopf algebraic case in \cite{alekseevmoduli,agsmoduli,asmoduli}  and more generally in  \cite[Section~5.4]{bzbj}.

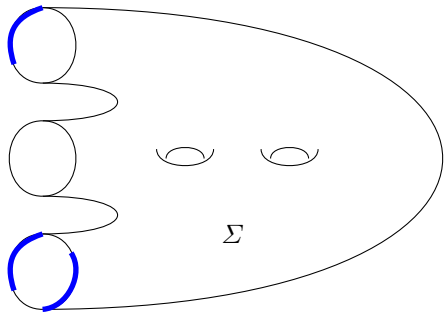
\begin{figure}[h]
	\centering
\begin{tikzpicture}
	\begin{pgfonlayer}{nodelayer}
		\node [style=none] (0) at (-7, 3) {};
		\node [style=none] (1) at (-7, 1) {};
		\node [style=none] (2) at (-7, 0) {};
		\node [style=none] (3) at (-7, -2) {};
		\node [style=none] (4) at (-7, -3) {};
		\node [style=none] (5) at (-7, -5) {};
		\node [style=none] (6) at (-3.75, -1) {};
		\node [style=none] (7) at (-2.75, -1) {};
		\node [style=none] (10) at (-4, -0.75) {};
		\node [style=none] (11) at (-2.5, -0.75) {};
		\node [style=none] (12) at (-1, -1) {};
		\node [style=none] (13) at (0, -1) {};
		\node [style=none] (14) at (-1.25, -0.75) {};
		\node [style=none] (15) at (0.25, -0.75) {};
		\node [style=none] (16) at (-2, -3) {$\Sigma$};
		\node [style=none] (17) at (-7, 3) {};
		\node [style=none] (18) at (-7.75, 1.5) {};
		\node [style=none] (19) at (-7, -3) {};
		\node [style=none] (20) at (-7.75, -4.5) {};
		\node [style=none] (21) at (-6.25, -3.5) {};
	\end{pgfonlayer}
	\begin{pgfonlayer}{edgelayer}
		\draw [bend right=90, looseness=1.50] (0.center) to (1.center);
		\draw [bend left=90, looseness=1.50] (0.center) to (1.center);
		\draw [bend right=90, looseness=1.50] (2.center) to (3.center);
		\draw [bend left=90, looseness=1.50] (2.center) to (3.center);
		\draw [bend right=90, looseness=1.50] (4.center) to (5.center);
		\draw [bend left=90, looseness=1.50] (4.center) to (5.center);
		\draw [bend left=90, looseness=6.75] (1.center) to (2.center);
		\draw [bend left=90, looseness=6.75] (3.center) to (4.center);
		\draw [bend left=90, looseness=4.50] (0.center) to (5.center);
		\draw [bend left=90] (6.center) to (7.center);
		\draw [bend right=90] (10.center) to (11.center);
		\draw [bend left=90] (12.center) to (13.center);
		\draw [bend right=90] (14.center) to (15.center);
		\draw [style=open, bend right=45, looseness=1.25] (17.center) to (18.center);
		\draw [style=open, bend right=45, looseness=1.25] (19.center) to (20.center);
		\draw [style=open, bend left=60] (21.center) to (5.center);
	\end{pgfonlayer}
\end{tikzpicture}
\label{figsigma}
\caption{A compact oriented surface $\Sigma$ with three parametrized intervals (in blue)
in its boundary.}
\end{figure}

The main result of this article is that $\int_\Sigma \cA$ is a pivotal $\cA^{\boxtimes n}$-module if $\cA$ is a \emph{unimodular} finite ribbon category.
Recall that $\cA$ is called \emph{unimodular} if the distinguished invertible object $\alpha \in \cA$ controlling the quadruple dual via $-^{\vee\vee\vee\vee}\cong \alpha \otimes -\otimes \alpha^{-1}$~\cite{eno-d} is isomorphic to the monoidal unit.

\begin{reptheorem}{thmmain}
Let $\cat{A}$ be a unimodular finite ribbon category with chosen trivialization $\alpha \cong \mathbbm{1}$ of the distinguished invertible object,
	and let $\Sigma$ be
	a compact, oriented surface  with $n$ marked boundary intervals, at least one per connected component.
	Then the factorization homology $\int_\Sigma \cat{A}$
	comes with the structure of a pivotal left
	 $\cat{A}^{\boxtimes n}$-module category. 
\end{reptheorem}

\subsection*{Connection to full open conformal field theories} 

In order to explain how this result can be applied in conformal field theory, let us introduce some terminology.  The \emph{open surface operad} $\mathsf{O}$~\cite{giansiracusa,sn,envas} is the groupoid-valued operad whose groupoid $\mathsf{O}(n)$ of operations of arity $n$ has as
 objects compact, oriented surfaces with at least one boundary component per connected component together with $n+1$ parametrized intervals embedded in their boundary; the morphisms are mapping classes
 (isotopy classes of diffeomorphisms preserving the orientation and the parametrized intervals)
  and operadic composition is given by gluing surfaces along the parametrized intervals. The open surface operad is in fact a \emph{modular operad} in the sense of \cite{gkmod}.

 A  \emph{two-dimensional categorified open topological field theory}~\cite{envas}, also referred to as an \emph{open modular functor},
  is defined as a  modular $\mathsf{O}$-algebra that (at least in the context of this article) takes values  in the symmetric monoidal bicategory $\Lexf$ of finite categories, left exact functors and natural transformations. Such a structure is exactly the mathematical description
  of the monodromy data of an \emph{\underline{open} conformal field theory}; this is informed by the notions
  \cite{Lazaroiu,costellotcft,ffrsunique} 
   and defined in this form \cite[Section~8]{microcosm} in the context of the modular microcosm principle.

  After unpacking the definition, a $\Lexf$-valued
  	 open modular functor consists of the data of a finite category $\cA$ and a rule that associates to every surface $\Sigma$ with at least one boundary component per connected component, $n$ marked boundary intervals (again as in Figure~\ref{figsigma}) and $n$ objects $X_1, \ldots, X_n \in \cA$ attached to the boundary intervals, a vector space, the so-called \emph{space of (open) conformal blocks}, which carries an action of the mapping class group of the surface and is compatible with gluing along parametrized intervals. By definition, the space of conformal blocks is left exact in the boundary labels.

According to \cite[Theorem 2.2]{envas},  a $\Lexf$-valued open modular functor is, by evaluation on disks with marked intervals in their boundary, equivalent to the datum of a \emph{pivotal Grothendieck-Verdier category} in $\Lexf$ in the sense of \cite{bd}. In particular, the underlying pivotal finite tensor category of a finite ribbon category $\cA$ extends to a unique open modular functor that we denote by $\Ao$. 

Within the framework of open modular functors, it is a natural question to ask about correlators.   A \emph{consistent system of correlators}
for the open conformal field theory with monodromy data $\cA$ amounts to a self-dual object $F\in\cA$ plus a family 
of vectors $\xi_\Sigma^F \in \Ao(\Sigma; F, \ldots , F) $ inside the spaces of open conformal blocks for all surfaces $\Sigma$
that is fixed by the action of the mapping class group and that solves the  so-called \emph{sewing constraints}, meaning  that the family is compatible with gluing along the parametrized boundary intervals. For more background, we refer to
 \cite{frs25, algcften}. Note that what we describe here are correlators for an open conformal field theory; these are different from so-called bulk field correlators, even though both notions are intimately connected \cite{ffrsunique,correlators}.
 The pair $(\Ao , \xi^F)$ 
 should be thought of as a \emph{full} open conformal field theory; $\Ao$ describes 
  the monodromy data of the open conformal field theory while the consistent system of correlators is a mathematically precise incarnation of the \emph{solution} of the theory, see e.g.~the introductions of \cite{csrcft,fspivotal,rcftsn,correlators} for more background.

One of the main results of \cite{microcosm} is that consistent systems of open correlators for $\Ao$ are in one-to-one correspondence with symmetric Frobenius algebras in $\cA$
(this remains true for Grothendieck-Verdier categories, even through we do not need it here), extending the corresponding result in the  semisimple rigid case  \cite{ffrsunique}. We use this to prove that for a 
compact, connected, oriented surface $\Omega$ with a single marked boundary interval,
the \emph{moduli algebra} $\mathfrak{a}_\Omega:= \END(\cat{O}_\Omega^\cA)\in\cA$
from~\eqref{eqnmodulialgebraxy}
inherits a symmetric Frobenius structure and extends to a system of open correlators of geometric origin:

\begin{repcorollary}{coropen}
	Let $\cA$ be a unimodular finite ribbon category and denote by $\Ao$ the unique open modular functor whose evaluation on disks with marked boundary intervals is given by $\cA$ as pivotal finite tensor category.
	Fix a compact, connected, oriented surface $\Omega$ with a single marked boundary interval. 
	Then the moduli algebra $\mathfrak{a}_\Omega$ gives rise to a consistent system of open correlators
$$ \xi_\Sigma^{\Omega} \in \Ao(\Sigma; \mathfrak{a}_\Omega,\dots,\mathfrak{a}_\Omega) \quad \text{for all}\quad \Sigma \in \O(n)$$
	that define a full open conformal field theory $(\Ao,\xi^\Omega)$. 
	Explicitly, the datum of $ \xi_\Sigma^{\Omega}$ amounts equivalently to 
	\begin{pnum}
		\item a map
$$
		\cat{O}_\Sigma^\cat{A} \to \mathfrak{a}_\Omega^{\boxtimes n} \ogreaterthan\cat{O}_\Sigma^\cat{A}
$$
in $\int_\Sigma \cat{A}$
that is $\Map(\Sigma)$-invariant with respect to the  action of the mapping class group $\Map(\Sigma)$ which arises from the homotopy fixed point structure of $\cat{O}_\Sigma^\cat{A}$,\label{item:i_cor}

	\item a map 
$$
		\mathfrak{a}_\Omega^{\boxtimes n} \to \mathfrak{a}_\Sigma
$$ in $\cA^{\boxtimes n}$ that is invariant under the action of $\Map(\Sigma)$ on $\mathfrak{a}_\Sigma$ constructed by Alekseev-Grosse-Schomerus and more generally Ben-Zvi-Brochier-Jordan.  \label{item:ii_cor}
\end{pnum}
\end{repcorollary}

The open correlators can be explicitly calculated using the Frobenius structure of $\mathfrak{a}_\Omega$, see
\cref{excalculation}. The mapping class group $\Map(\Omega)$ of $\Omega$ acts by symmetries of the full open conformal field theory, see \cref{remsymmetry}.

\subsection*{Modular invariant modified traces}
\emph{Modified traces}~\cite{geerpmturaev} are a non-semisimple replacement for quantum traces and have become an indispensable tool in modern quantum topology, most notably for the construction of the Costantino-Geer-Patureau-Mirand manifold invariants~\cite{cgp}.
A unimodular finite ribbon category $\cA$ has a modified trace on its tensor ideal $\Proj \cA$ of projective objects~\cite{mtrace}, i.e.\
a family of linear maps
$$
\tr\!_P : \Hom_{\cA}(P,P)\to k
$$
indexed by $P \in  \Proj \cA$ that are 
\begin{itemize}
	\item \emph{cyclic}, i.e.\ they descend to a $k$-valued map
\begin{align}\label{eq:description_HH_intro}
	\tr : H\! H_0(\cA) := \left.\bigoplus_{P\in \Proj \cA }
	\Hom_{ \cA}(P,P) \right/\! 	\! \!   \catf{span}(g\circ f- f\circ g)  \  \to k \ , 
	\end{align}
where we divide out by the subspace spanned by $g\circ f- f\circ g$  for morphisms $f:P\to Q$ and $g:Q\to P$ between projective objects,
\item \emph{non-degenerate}, i.e.\
the pairings
\begin{align}
	\Hom_{ \cA}(P,Q) \otimes \Hom_{ \cA}(Q,P) \to k \ , \quad f \otimes g \mapsto \tr\!_P (g\circ f)=\tr\!_Q(f\circ g)
\end{align} are non-degenerate for all $P,Q \in \Proj\cA$,
\item and have the \emph{partial trace property}, i.e.\
for $X\in \cA$ and $P\in\Proj\cA$ we have
\begin{align}
	\tr\!_{X\otimes P}(    f      ) = \tr\!_{ P} (p_X(f)) 
	\end{align}
for every map $f:X\otimes P\to X\otimes P$, with $p_X(f):P \to P$ being the \emph{partial trace} of $f$ defined as the composite 
$$ P \ra{\widetilde b_X\otimes \id_P} {^\vee }X \otimes X \otimes P \ra{\vartheta \otimes f} X^\vee \otimes X \otimes P \ra{d_X\otimes \id_P} P  $$
for 
\begin{itemize}
	\item the coevaluation $\widetilde b_X: I\to {^\vee }X \otimes X$,
	\item the identification $\vartheta : {^\vee }X \cong X^\vee  $ induced by the pivotal structure,
	\item the evaluation $d_X:X^\vee \otimes X \to I$.
	\end{itemize}
(This is the left partial trace property; an analogue on the right holds automatically by \cite[Corollary 6.12]{shibatashimizu} because $\cA$ is ribbon.)
\end{itemize}
This modified trace on $\cA$ is in fact unique up scalar multiple, and a  choice of a modified trace amounts to a choice of trivialization $\alpha \cong \mathbbm{1}$ of the distinguished invertible object of $\cA$.

Since the factorization homology $\int_\Sigma \cA$ 
is the result of `integrating' $\cA$ over $\Sigma$, the following question is natural: \emph{Can we also integrate the modified trace of $\cA$ over $\Sigma$  such that we recover the original modified trace whenever $\Sigma$ is a disk?} In more detail, one may ask
whether we can construct trace functions
\begin{align}
	\tr\!^\Sigma : H \! H_0\left(\int_\Sigma \cA\right)=\left.\bigoplus_{P\in \Proj \int_\Sigma \cA }
	\Hom_{\int_\Sigma \cA}(P,P) \right/\! 	\! \! {\sim} \  \to k   \label{eqntracefunctionsi}
\end{align} satisfying the analogous properties of the ones mentioned above in the module category setting, and maybe even more. Most crucially, since the mapping class group $\Map(\Sigma)$ acts on $H \! H_0\left(\int_\Sigma \cA\right)$, we should ask whether the traces  can be made invariant under the mapping class group action.
We use \cref{thmmain} to show	 that we can indeed obtain such modular invariant traces:

\begin{reptheorem}{thmmodtrace}[Modular invariant modified traces]
Let $\cat{A}$ be a unimodular finite ribbon category with a chosen trivialization $\alpha \cong \mathbbm{1}$ of the distinguished invertible object,
	and let $\Sigma$ be
	a compact, oriented surface  with $n$ parametrized boundary intervals, at least one per connected component.
	Then there are trace maps 
	\begin{align}
		\tr\!^\Sigma : H \! H_0\left(\int_\Sigma \cA\right)=\left.\bigoplus_{P\in \Proj \int_\Sigma \cA }
		\Hom_{\int_\Sigma \cA}(P,P) \right/\! 	\! \! {\sim} \  \to k   \label{eqntracefunctions}
	\end{align} originating from the pivotal $\cA^{\boxtimes n}$-module structure from \cref{thmmain}
that have the following properties:
\begin{enumerate}
	\item[(N)] The traces are non-degenerate in the sense that the pairings
$$
		\Hom_{\int_\Sigma \cA}(P,Q) \otimes \Hom_{\int_\Sigma \cA}(Q,P) \to k \ , \quad f \otimes g \mapsto \tr\!^\Sigma (g\circ f)=\tr\!^\Sigma(f\circ g)
$$ for $P,Q \in \Proj \int_\Sigma \cA$ are non-degenerate.
	\item[(M)] The traces are modular invariant, i.e.\ preserved by the action of the mapping class group $\Map(\Sigma)$ on $H \! H_0\left(\int_\Sigma \cA\right)$.
	\item[(L)] The traces are local, i.e.\ if $\Sigma'$ arises from $\Sigma$ by gluing a pair of marked intervals together, then the excision equivalence $\cat{A}	\boxtimes_{\cA^\env} \int_\Sigma \cat{A} \simeq \int_{\Sigma'} \cat{A}$ preserves the traces (on the left side, we have the traces originating from Proposition~\ref{propgluing2}).
	
	\item[(P)] The traces have the higher genus partial trace property, i.e.\ for each $X\in\cA^{\boxtimes n}$, $P\in\Proj \int_\Sigma \cA$ and $f: X \ogreaterthan P \to  X \ogreaterthan P$, we have
$$
		\tr\!^\Sigma(f)=\tr\!^\Sigma(\close_{X|P}(f)) \ , 
$$
	for the Shibata-Shimizu closing operator
$$ \close_{X|P}:\Hom_{\int_\Sigma \cA}(X \ogreaterthan P,X \ogreaterthan P )\to \Hom_{\int_\Sigma \cA}( P, P )
$$ for the pivotal $\cA^{\boxtimes n}$-module $\int_\Sigma \cA$.
	\end{enumerate}
The restriction to marked disks amounts to a two-sided modified trace on $\Proj \cA$. Conversely, given a 
 modified trace on $\Proj \cA$, it extends to a unique family of trace functions on $\int_\Sigma \cA$
  satisfying (N) and (L); the conditions (M) and (P) are then automatic.
\end{reptheorem}

The condition formulated using the closing operator introduced by Shibata-Shimizu~\cite[Section~5.1]{shibatashimizu} amounts to a relative version of the partial trace property and 
will be recalled in \cref{sectraceHH}.

Theorem~\ref{thmmodtrace} implies in the skein-theoretic picture for factorization homology~\cite{cooke,brownhaioun,mwskein} that skein categories of unimodular finite ribbon categories come with modular invariant
modified traces. However, we do not see a proof directly relying on skein theory. Most certainly, proving this on generators of the mapping class group does not seem doable.
In fact, the argument for  Theorem~\ref{thmmodtrace} reads the construction from topology to correlation functions used for \cref{coropen} \emph{backwards}, i.e.\ it leverages the correlator construction for a purely mathematical result: This should not come as a surprise because in Theorem~\ref{thmmodtrace} we try to build a modular invariant quantity that is compatible with gluing. 
In slightly more technical terms, we use that the pivotal structure in Theorem~\ref{thmmain} (and  the associated traces)
 is constructed from a gluing procedure for a ribbon graph model of the surface, but is ultimately independent of this ribbon graph model. 
This means that a mapping class will map the traces in \cref{thmmodtrace} built in reference to a ribbon graph model
 to the traces built
 in reference to a different ribbon graph model (the one obtained by acting with the mapping class).
 But this is the \emph{same} trace because of the independence of the ribbon graph model. 
This is exactly the approach  towards the construction of correlation functions taken in \cite{microcosm,correlators} generalizing the classical constructions~\cite{frs1,ffrsunique}.
	
		\vspace*{0.2cm}\textsc{Acknowledgments.} 
We thank Adrien Brochier, Patrick Kinnear, Lukas Müller and Christoph Schwei\-gert for helpful discussions related to the manuscript.
		JB and LW  gratefully acknowledge support
		by the ANR project CPJ n°ANR-22-CPJ1-0001-01 at the Institut de Mathématiques de Bourgogne (IMB).
		The IMB receives support from the EIPHI Graduate School (ANR-17-EURE-0002).

\section{Preliminaries}
In this section, we review some basic algebraic notions related to module categories over finite tensor categories as well as some generalities on factorization homology.  Unless otherwise stated, all categories, functors and natural transformations  will be assumed to \emph{$k$-linear} over a fixed algebraically closed  field $k$.

\subsection{Finite tensor categories}

Following \cite{etingofostrik}, a \emph{finite category} is an abelian category $\cA$ with finite-dimensional morphism
spaces, enough projective objects, finitely many isomorphisms classes of simple objects and where any object has finite length. An abelian category is finite if and only if it equivalent to the category $A$-$\catf{mod}$ of finite-dimensional modules over a finite-dimensional algebra $A$.

For a functor $F: \cA \to \cB$  between finite categories, 
\begin{equation}\label{lem:left_exact=right_adjoint}
\text{$F$ is left exact } \iff  \text{$F$ is a right adjoint} \qquad , \qquad  \text{$F$ is right exact } \iff  \text{$F$ is a left adjoint}\ , 
\end{equation} 
see e.g.~\cite[Corollary 2.3]{fss}. A finite category $\cA$ equipped with a rigid bilinear monoidal  product  with simple unit is called a \emph{finite tensor category}~\cite{etingofostrik}. For an object $X \in \cA$, we write $X^\vee$ and ${}^\vee X$ for its left and right dual, respectively.  There are adjunctions 
\begin{equation}\label{eq:adjunctions_duals_ftc}
X^\vee \otimes(-)\dashv X\otimes(-)\dashv {}^\vee X\otimes(-)\qquad , \qquad (-)\otimes {}^\vee X\dashv(-)\otimes X\dashv(-)\otimes X^\vee 
\end{equation}
e.g.\ \cite[Section 2.7]{shimizuunimodular},
so by  \eqref{lem:left_exact=right_adjoint} the monoidal product functor $\otimes : \cA \times \cA \to \cA $ is exact in each variable.

Given a finite tensor category $\cA = (\cA , \otimes, \mathbbm{1})$, we will write  
$ \cA^{\mathsf{op}} = (\cA^{\mathsf{op}}, \otimes, \mathbbm{1}) $ and  $ \cA^{\rev} = (\cA, \otimes^{\op}, \mathbbm{1}) $, where $
X \otimes^{\op} Y := Y \otimes X$ for $X,Y\in \cA$.

\subsection{Finite module categories and internal homs}\label{subsec:finite_mod_cats}

Let $\cC$ be a monoidal category. A \emph{left $\cC$-module structure} on a category $\cM$ is the data of an action functor $\ogreaterthan : \cC \times \cM \to \cM $ and a family of natural isomorphisms
$\mathbbm{1}\ogreaterthan M \cong M $ and $ (X \otimes Y)\ogreaterthan  M \cong X \ogreaterthan (Y \ogreaterthan M)  $ for $X,Y\in\cA$ and $M\in\cM$
satisfying  pentagon and triangle axioms similar to those for monoidal categories, see \cite[Section 7.1]{egno} for details. A \emph{right $\cA$-module structure} and an \emph{$\cA$-$\cB$-bimodule structure} are defined analogously.  Similarly, an \emph{$\cA$-module functor} $F: \cM \to \mathcal{N}$ between left $\cA$-module categories is a functor together with natural isomorphisms $ F(X \ogreaterthan M) \cong F(X) \ogreaterthan M   $ 
satisfying similar axioms \cite[Section 7.2]{egno}.

If $\cA$ is a finite tensor category, a finite category $\cM$ equipped with a  $\cA$-module structure is called a \emph{finite $\cA$-module category} if the functors $- \ogreaterthan M : \cA \to \cM$  are right exact for any $M \in \cM$. Similarly to \eqref{eq:adjunctions_duals_ftc}, there are adjunctions
\begin{equation}\label{eq:adjunctions_action_functor}
X^\vee \ogreaterthan (-)\dashv X\ogreaterthan (-)\dashv {}^\vee X\ogreaterthan(-)
\end{equation} 
so again  \eqref{lem:left_exact=right_adjoint} implies that the functor $X \ogreaterthan - : \cM \to \cM$ is always exact for any $X \in \cA$. This implies that 
 $X \ogreaterthan P$ for any $X \in \cA$ is projective whenever $P \in \cM$ is projective.

Let $\cA$ be a finite tensor category and $\cM$ a finite $\cA$-module category. For every $M \in \cM$, the functor $-	 \ogreaterthan M$ has a right adjoint by \eqref{lem:left_exact=right_adjoint}, that we denote by $\HOM(M,-): \cM \to \cA$. By definition, it comes with linear isomorphisms
\begin{equation}\label{eq:internal_hom_adj}
\Hom_{\cM}(X \ogreaterthan M,N) \cong \Hom_\cA (X, \HOM(M,N)) 
\end{equation}
natural in $X \in \cA$ and $N \in \cM$.
According to \cite[Theorem IV.7.3]{maclane}, the right adjoints in this family of adjunctions assemble in a unique way into a functor $\HOM:\cM^{\mathsf{op}} \times \cM \to \cA$
such that the isomorphism \eqref{eq:internal_hom_adj} is natural in the three variables. One calls $\HOM$ the \emph{internal hom functor}. This functor  is, just like the hom functor,  left exact in both variables. Also, for objects $X,Y \in \cA$ and $M,N \in \cM$,
 there is a natural isomorphism
\begin{equation}\label{eq:iso_2.15}
\HOM(X \ogreaterthan M, Y \ogreaterthan N) \cong Y \otimes \HOM(M,N) \otimes X^\vee,
\end{equation}
exhibiting $\HOM$ as an $\cA$-bimodule functor, 
see \cite[Lemma 7.9.4]{egno}.

The aforementioned internal homs equip $\cM$ with a structure of $\cA$-enriched category. In particular, for any object $M \in \cM$, the object $\END(M):= \HOM(M,M)$ is an algebra in $\cA$, and for any  $N \in \cM$ the object $\HOM(M,N)$ is a right $\END(M)$-module in $\cA$. We denote by $\mathsf{mod}_\cA{\text -}\END(M)$ the category of right $\END(M)$-modules in $\cA$. It is readily verified that the monoidal product turns it into  a left $\cA$-module category.

Recall that an object $M \in \cM$ is called \emph{$\cA$-projective} (respectively \emph{$\cA$-generator}, respectively \emph{$\cA$-progenerator}) if the functor $\HOM(M,-): \cM \to \cA$ is exact (respectively faithful, respectively exact and faithful). In the sequel we will make use of the following `monadicity theorem', e.g. \cite[Section 7.10]{egno}: If  $M \in \cM$, then the functor
\begin{equation}\label{eq:monadicity_thm}
\HOM(M,-): \cM \ra{\simeq}  \mathsf{mod}_\cA{\text -}\END(M) 
\end{equation}
is an equivalence of left $\cA$-module categories if and only if $M$ is an $\cA$-progenerator. Moreover, any finite $\cA$-module category has an $\cA$-progenerator.

\subsection{Balanced Deligne products}
Let $\cA$ be a finite tensor category and let $\cM$ and $\cN$ be finite right and left $\cA$-module categories, respectively. An \emph{$\cA$-balanced functor} from $\cM \times \cN$ to a finite category $\cE$  is a right exact functor $F: \cM \times \cN \to \cE$ together with a family of natural isomorphisms $ F(M \olessthan X, N) \cong F(M, X \ogreaterthan N)   $ satisfying  an analogue of the pentagon and triangle axioms. We denote by $\mathsf{Bal}_\cA (\cM \times \cN , \cE)$ the category of right exact $\cA$-balanced functors $\cM \times \cN \to \cE$ and natural transformations compatible with these natural isomorphisms.

If $\cA, \cM$ and $\cN$ are as above, their \emph{balanced Deligne product} is a finite category $\cM \boxtimes_\cA \cN$ together with a functor $
\boxtimes_\cA : \cM \times \cN \to \cM \boxtimes_\cA \cN$ 
which is right exact $\cA$-balanced and is universal with respect to this property, in the sense that precomposition with $\boxtimes_\cA $ induces an equivalence
$ \Rex (\cM \boxtimes_\cA \cN, \cE) \ra{\simeq} \mathsf{Bal}_\cA (\cM \times \cN , \cE) $
for any other finite category $\cat{E}$, see \cite{eno_fc} for this definition in the context of fusion categories. It is shown in \cite[Theorem 3.3]{dss} that the balanced Deligne product always exists, that the functor $\boxtimes_\cA $ is in fact exact and moreover that we have the following linear isomorphism:
\begin{equation}\label{eq:iso_DSS}
\Hom_{\cM \boxtimes_\cA \cN}(M \boxtimes_\cA N, M' \boxtimes_\cA N') \cong \Hom_\cA (\mathbbm{1}, \HOM(M,M') \otimes \HOM(N,N'))\ .
\end{equation}

It will be useful to have a concrete realization of the balanced Deligne product. For finite right and left $\cA$-module categories $\cM$ and $\cN$, choose equivalences 
 $\cM \simeq A \text{-} \mathsf{mod}_\cA$ and $ \cN \simeq \mathsf{mod}_\cA  \text{-} B$ for some  algebras  $A, B \in \cA$. According to \cite[Theorem 3.3]{dss}, the balanced Deligned product $\cM \boxtimes_\cA \cN $ can be realised as the category $A\text{-}\mathsf{mod}_\cA\text{-}B$ of $A$-$B$-bimodules in $\cA$; more precisely, the functor
\begin{equation}\label{eq:equiv_balanced_Deligne}
\otimes: A \text{-} \mathsf{mod}_\cA \boxtimes_\cA \mathsf{mod}_\cA  \text{-} B\ra{\simeq} A\text{-}\mathsf{mod}_\cA\text{-}B \qquad , \qquad M \boxtimes_{\cA} N \mapsto M \otimes N
\end{equation}
is an equivalence.  From this, we conclude the following statement, that we believe to be well-known  but whose proof we spell out in lack of a reference: 

\begin{corollary}\label{cor:projs_in_balanced_Deligne0}
	Let $\cA$ be a finite tensor category, $\cat{M}$ a finite right $\cA$-module category and $\cat{N}$ a finite left $\cA$-module category.
	Then the projective objects of $\cat{M}\boxtimes_\cA\cat{N}$ are direct summands of objects of the form $X\boxtimes Y$ with  objects $X\in \cat{M}$ and $Y\in\cat{N}$, where $X$ is projective. The same statement holds true for objects of the form $X\boxtimes Y$ with $Y$ projective.
\end{corollary}

\begin{proof}
	Let us choose identifications $\cM \simeq A \text{-} \mathsf{mod}_\cA$,  $ \cN \simeq \mathsf{mod}_\cA  \text{-} B$ and  $\cat{M}\boxtimes_\cA\cat{N} \simeq A\text{-}\mathsf{mod}_\cA\text{-}B $ for some algebras $A, B \in \cA$ as above. 
	As it is well-known, the projective objects of $A\text{-}\mathsf{mod}_\cA\text{-}B $ 
	are precisely the direct summands of the free modules $A\otimes P \otimes B$, where $P \in \cA$ is projective.
	Indeed, $A\otimes P \otimes B$ is projective because it is the image of the free bimodule functor $$A \otimes - \otimes B : \cA\to A\text{-}\mathsf{mod}_\cA\text{-}B ,$$  which preserves projective objects since its right adjoint,  the forgetful functor, is exact. 
	Conversely, given a projective $A$-$B$-bimodule $T$, consider a projective object $P$ in $\cA$ 
	with epimorphism $P\to T$. Applying the free bimodule functor we get an epimorphism  $A\otimes P\otimes B\to A\otimes T\otimes B$ since right exact functors preserve epimorphisms. The composite  $$B\otimes P\otimes C\to B\otimes T\otimes C\to T$$ with the action structure map (which is an epimorphism)  is again an  epimorphism in $A\text{-}\mathsf{mod}_\cA\text{-}B $.  As $T$ is projective, the epimorphism
	$B\otimes P\otimes C\to  T $ splits, thereby exhibiting $T$ as direct summand of  $A\otimes P\otimes B$. 
	
	This gives us the statement of the Lemma
	 since under the equivalence \eqref{eq:equiv_balanced_Deligne} the free bimodule  $A \otimes P \otimes B$  corresponds to the object  $(A \olessthan P) \boxtimes_{\cA} B \cong A \boxtimes_{\cA} (P  \ogreaterthan B)$ in $\cat{M}\boxtimes_\cA\cat{N}$, where 	 $A\olessthan P$ is a projective left $A$-module and  $P \ogreaterthan B$ is a projective right $B$-module.
\end{proof}

\subsection{The distinguished invertible element and unimodularity}
If $\cA$ is a finite tensor category, its \emph{envelope} is $\cA^\env := \cA \boxtimes \cA^\rev$, which is a finite tensor category by  {\cite[Proposition 5.17]{deligne}. The category $\cA$ can be viewed as a left $\cA^\env$-module category with action functor determined by 
$ (X \boxtimes Y) \ogreaterthan Z = X \otimes Z \otimes Y.  $
The algebra $\Delta :=\END(\mathbbm{1}) \in \cA^\env$ is often  called the \emph{canonical algebra} of $\cA$. In fact, according to \cite[Lemma 4.4]{shimizuunimodular}, $\Delta$ can be expressed as the coend $
\Delta \cong \int^{X \in \cA} X^\vee  \boxtimes X $  
(we refer the reader to \cite[Section~2.2]{fss} for an introduction to coends and their properties).
We choose the notation $\Delta$ for the canonical algebra because its underlying object happens to be also the coevaluation object for $\cA$ as cyclic associative algebra~\cite{cyclic} that is denoted by the same symbol.

By \cite[Proposition 4.15]{bzbj} monoidal unit $\mathbbm{1} \in \cA$ is  an $\cA^\env$-progenerator. By \eqref{eq:monadicity_thm}, $ \HOM(\mathbbm{1},-): \cA \ra{\simeq} \mathsf{mod}_{\cA^\env}\text{-} \Delta $
is an equivalence of $\cA^\env$-module categories. 
The unique (up to isomorphism)  object $\alpha \in \cA$  such that $
(\alpha \boxtimes \mathbbm{1}) \otimes \Delta \cong \Delta^\vee$ is called the \emph{distinguished invertible object}. By \cite[Theorem 3.3]{eno-d}, one may describe the quadruple dual through an isomorphism $(-)^{\vee \vee \vee \vee} \cong \alpha \otimes - \otimes \alpha^{-1}$. Note that if $\cA$ is unimodular, the isomorphism $\alpha \cong \mathbbm{1}$ induces an
isomorphism
\begin{equation}\label{eqnpsi}
\psi : \Delta \ra{\cong} \Delta^\vee.
\end{equation} 
If $\cA$ comes with a pivotal structure $\omega : \id_{\cA}\cong-^{\vee\vee}$, then it induces a pivotal structure on $\cA^\env$ also denoted by $\omega$ such that the composite   $\Delta \ra{\omega_\Delta}\Delta^{\vee \vee} \ra{\psi^\vee} \Delta ^\vee$ agrees with $\psi$.

\subsection{Nakayama functors}

Let $\cA, \cB$ be finite categories. By the Morita invariant version of the Eilenberg-Watts theorem given in \cite[Theorem 3.2]{fss}}, there is a zig-zag of adjoint equivalences
$$
\Lex(\cA,\cB) \xleftarrow{\ \simeq\ }  \cA^\op \boxtimes \cB \ra{\simeq} \Rex(\cA,\cB)\ .
$$
Because the identity functor of $\cA$ is both left and right exact, the above zig-zag gives rise to two canonical functors,
namely the \emph{left Nakayama functor} $\N^\ell \in \Lex(\cA, \cA)$ and the \emph{right Nakayama functor} $\N^r \in \Rex(\cA, \cA)$. By  \cite[Proposition 3.20]{fss} we have
\begin{equation}
\N^\ell_{\cA \boxtimes \cB} \cong \N^\ell_{\cA } \boxtimes \N^\ell_{\cB} \qquad , \qquad  \N^r_{\cA \boxtimes \cB} \cong \N^r_{\cA } \boxtimes \N^r_{\cB}
\end{equation}
for finite categories $\cA$ and $\cB$. Also, by \cite[Lemma~4.7]{shibatashimizu} or \cite[Corollary~2.3]{tracesw}, for any finite category $\cA$ and projective objects $P,Q \in \cA$, there are  natural isomorphisms
\begin{equation}\label{eq:Nr-twisted_CY}
\Hom_{\cA} (P,Q) \cong \Hom_{\cA}(Q, \N^r P)^*  \ .
\end{equation}

Moreover, if $\cM$ is a finite left module category over a finite tensor category $\cA$, then thanks to \cite[Proposition 3.18]{fss} the left and right Nakayama functors of $\cM$ have a natural structure of $(-)^{\vee \vee}$-twisted and ${}^{\vee \vee}(-)$-twisted $\cA$-module functors, respectively. Writing $G=(-)^{\vee \vee}$ or $  {}^{\vee \vee}(-)$, this means that they can be viewed as module functors $\cM \to {}_{G}\cM$, where ${}_{G}\cM$ stands for $\cM$ with the module structure twisted by $X \ogreaterthan_G M := G(X) \ogreaterthan M$. Explicitly, there are natural isomorphisms 
\begin{equation}\label{eq:Nakayama_twisted}
\N^\ell(X \ogreaterthan M) \cong  X^{\vee \vee} 	\ogreaterthan  \N^\ell( M) \qquad , \qquad \N^r(X \ogreaterthan M) \cong  {}^{\vee \vee} X	\ogreaterthan  \N^r( M)\ ,
\end{equation}
satisfying the corresponding coherence conditions, see \cite[Theorem 4.4]{fss} and \cite[Section~2.5]{relserre}.

\subsection{Relative Serre functors}

Let $\cA$ be a finite tensor category and let $\cM$ be a finite left $\cA$-module category. A \emph{left} and \emph{right} \emph{relative Serre functor} of $\cM$~\cite{schaumannpiv,fss}
are endofunctors $\mathsf{S}^\ell, \mathsf{S}^r : \cM \to \cM$  together with a family of natural isomorphisms in $\cA$ 
\begin{equation}\label{eq:Serre_natiso}
\HOM (\mathsf{S}^\ell(M),N) \cong {}^\vee \HOM(N,M) \qquad  , \qquad \HOM (M,\mathsf{S}^r(N)) \cong \HOM(N,M)^\vee \ .
\end{equation}
If they exist, then they are essentially unique by the Yoneda Lemma. It is shown in \cite[Proposition 4.24]{fss} that a finite left $\cA$-module category $\cM$ has left and right relative Serre functors if and only if $\cM$ is \emph{exact}, i.e.\ if all objects  are $\cA$-projective. 

\begin{example}\label{ex:Serre_funct_in_ftc}
Let $\cA$ be a finite tensor category viewed as a finite left module category over itself. Then  $\sS^\ell = {}^{\vee \vee}(-)$ and  $\sS^r = (-)^{\vee \vee}$ are readily seen to be  the left and right relative Serre functors (in particular $\cA$ is exact). This says that these functors can be interpreted as the analogs of the double left and right dual functors in module categories.
\end{example}

According to \cite[Lemma 4.23]{fss} and \cite[Section~3.3]{relserre}, a left or  right relative Serre functor of $\cM$ has a unique structure of twisted $\cA$-module functor such that the natural isomorphisms \eqref{eq:Serre_natiso}  are isomorphisms of $\cA$-bimodule functors. This implies that there are natural isomorphisms (compare with \eqref{eq:Nakayama_twisted})
\begin{equation}\label{eq:Serre_twisted}
\sS^\ell(X \ogreaterthan M) \cong {}^{\vee \vee} X 	\ogreaterthan  \sS^\ell( M) \qquad , \qquad \sS^r(X \ogreaterthan M) \cong  X^{\vee \vee} 	\ogreaterthan  \sS^r( M)\ . 
\end{equation}
The relative Serre functors are in fact very closely related to the Nakayama functors: If $\cM$ is an exact left $\cA$-module category over a    finite tensor category $\cA$ with distinguished invertible object $\alpha$, there are natural isomorphisms 
 \begin{equation}\label{eq:N=alpha.S}
 \N^\ell \cong \alpha \ogreaterthan  \sS^\ell \qquad , \qquad  \N^r \cong \alpha^{-1} \ogreaterthan  \sS^r 
 \end{equation}
of twisted module functors \cite[Theorem 4.26]{fss}.
In particular, for the left and right Nakayama functor of $\cA$ itself
we have natural isomorphisms
\begin{equation}\label{eq:N=alpha.(-)**}
\N^\ell \cong \alpha \otimes  {}^{\vee \vee} (-) \qquad , \qquad  \N^r \cong \alpha^{-1} \otimes  (-)^{\vee \vee}\ . 
\end{equation}

From this, the following can be read off: If $\cA$ and $\cB$ are finite tensor categories, the distinguished invertible object $\alpha_{\cA \boxtimes \cB}$ of $\cA \boxtimes \cB$ (which is a finite tensor category by {\cite[Proposition 5.17]{deligne}) and that of $\cA^\rev$ are given by 
\begin{equation}\label{eq:disting_inv_for_Deligne_product}
\alpha_{\cA \boxtimes \cB} \cong \alpha_\cA \boxtimes \alpha_\cB \qquad , \qquad \alpha_{\cA^\rev} \cong \alpha_\cA\ ,
\end{equation}
where $\alpha_\cA$, $\alpha_\cB$ and $\alpha_{\cA^\rev}$ denote the distinguished invertible objects of $\cA$, $\cB$ and  $\cA^\rev$, respectively.

\subsection{Pivotal module structures}\label{sectraceHH}

Let $\cA$ be a pivotal finite tensor category, let $\cM$ be an exact left $\cA$-module category and let $\sS^r$ be a right relative Serre functor of $\cM$. We regard $\sS^r$ as an (untwisted) $\cA$-module functor with structure morphism $$\sS^r(X \ogreaterthan M) \cong  X^{\vee \vee} 	\ogreaterthan  \sS^r( M) \cong X 	\ogreaterthan  \sS^r( M)\ ,$$ where the first isomorphism is \eqref{eq:Serre_twisted} and the second is induced by the pivotal structure of $\cA$. The same observation applies to the left relative Serre functor and the Nakayama functors.

According to \cite[Section~3.6]{relserre},
a \emph{pivotal structure} for $\cM$ is a trivialization  $\sS^r \cong \id_\cM  $ of the right relative Serre functor of $\cM$ as $\cA$-module functor. An exact  left $\cA$-module category $\cM$ endowed with a pivotal structure will be called a \emph{pivotal left $\cA$-module category}.

\begin{example}\label{exregularmodule}
Let $\cA$ be a pivotal finite tensor category, viewed as a left $\cA$-module category (by left multiplication). By \cref{ex:Serre_funct_in_ftc}, $\cA$ is exact and the left and right relative Serre functors can be trivialized by means of the pivotal structure. That this is a trivialization as module functors follows from the fact that a pivotal structure is a monoidal natural isomorphism. Hence this yields a pivotal module structure.
\end{example}

From~\eqref{eq:N=alpha.S}, we conclude the following fact that we will need in the sequel:

\begin{lemma}\label{lemmapivstr}
		Let $\cA$ be  a unimodular pivotal finite tensor category with a fixed trivialization
	$\alpha \cong \mathbbm{1}$ of the distinguished invertible object, and let $\cM$ be an exact left $\cA$-module category. Then
	a pivotal left $\cA$-module structure on $\cM$
	amounts precisely to a trivialization $\N^r \cong \id_\cM$ as $\cA$-module functor.
\end{lemma}

\subsection{A brief review of factorization homology}\label{subsec:review_FH}

Factorization homology \cite{higheralgebra,AF} is an invariant of $n$-dimensional manifolds that satisfies a generalization of the classical Eilenberg-Steenrood axioms for homology theories. In the present paper, we will only be interested in factorization homology $\int_\Sigma \cA$ of  an oriented surface $\Sigma$ with coefficients in a framed $E_2$-algebra  $\cA$  in the symmetric monoidal $(2,1)$-category $\Rexf$ of finite categories, right exact functors and natural isomorphisms, with the Deligne product as monoidal product; that is, $\cA$ is a certain balanced braided finite category.

The factorization  homology  $\int_\Sigma \cA$ is defined as  the homotopy colimit (or $(2,1)$-colimit)
\begin{equation*}
	\int_\Sigma \cA := \hocolimsub{\substack{
			( D^2)^{\sqcup n} \hookrightarrow \Sigma   \\ n\ge 0}} \cA^{\boxtimes n}
\end{equation*}
over all disk embeddings $( D^2)^{\sqcup n} \hookrightarrow \Sigma $ and all $n\ge 0$.

Here is a list of important properties of factorization homology, see \cite{AF,ginot,bzbj} for more background:

\begin{pnum}
	\item There is an equivalence $\int_{D^2} \cA \simeq \cA$.
	
	\item There is an equivalence $\int_{\bar{\Sigma}} \cA \simeq \int_\Sigma \bar{\cA}$  where $\bar{\Sigma}$ denotes a surface $\Sigma$ equipped with the opposite orientation and $\bar{\cA}$ is the balanced braided monoidal category endowed with the same monoidal product, but inverse braiding and  inverse balancing.

	\item Factorization homology is a \emph{canonically pointed} theory in the following sense: For a surface $\Sigma$, the  embedding $\emptyset \hookrightarrow \Sigma$ induces a functor $ \cat{O}_\Sigma^\cat{A}: \vect \simeq \int_{\emptyset} \cA \to \int_\Sigma \cA  $ that of course amounts 
	to an 
	object $\cat{O}_\Sigma^\cA\in\int_\Sigma\cA$ that is called the \emph{quantum structure sheaf} in \cite{bzbj}.

	\item If $( D^2)^{\sqcup n} \hookrightarrow \Sigma $ is an embedding of a disjoint union of disks which factors through a bigger disk, then the corresponding universal morphism $\cA^{\boxtimes n} \to \int_\Sigma \cA$ factors through the $n$-fold monoidal product functor ($E_2$-multiplication) $\cA^{\boxtimes n} \to \cA$.\label{propertyproduct}

	\item Given a compact, oriented 1-manifold $P$, stacking induces a canonical monoidal structure on $\int_{P \times [0,1]} \cA$  with unit $\cat{O}_{ P \times [0,1]}^\cA$.

	\item Given a surface $\Sigma$ with boundary $\partial \Sigma$, the category $\int_\Sigma \cA$ is naturally a left $\int_{\partial \Sigma \times [0,1]} \cA$-module, induced by choosing a  collar neighborhood of $\partial \Sigma$.
	
	\item In complete analogy to the Eilenberg-Steenrood axioms, factorization homology satisfies a local-to-global principle called \emph{excision}, which is the main computational tool for factorization homology. Suppose $\Sigma_1, \Sigma_2$ are compact, oriented surfaces  let $\Sigma:= \Sigma_1 \cup_{P \times [0,1]} \Sigma_2$ be the surface obtained by a collar gluing along a 1-manifold $P$.  We choose the monoidal structure on  $\int_{P \times [0,1]} \cA$ so that $\int_{\Sigma_2} \cA$ is a   left $\int_{P \times [0,1]} \cA$-module category, and   $\int_{\Sigma_1} \cA$ is therefore a  right $\int_{P \times [0,1]} \cA$-module category.
	The excision property then tells us that the embedding $\Sigma_1 \sqcup \Sigma_2 \hookrightarrow \Sigma$ induces an equivalence
	\begin{equation}\label{eq:excision}
		\int_{\Sigma_1} \cA \boxtimes_{\int_{P \times [0,1]} \cA } \int_{\Sigma_2} \cA \simeq \int_{\Sigma} \cA.
	\end{equation} In fact, factorization homology
	is characterized by $\int_{D^2} \cA \simeq \cA$ and excision, see
	\cite[Section~3.3]{AF} for the precise statement.  
\end{pnum}

\subsection{Factorization as a module category}

Let $\cA$
be a   balanced braided  category $\cA$ in $\Rexf$ and suppose that $\Sigma$ is an oriented, compact surface $\Sigma$ 
with at least one boundary component per connected component and
 $n$ parametrized intervals in   $\partial \Sigma$ that we call \emph{marked boundary intervals}.
Here by a parametrization we understand an orientation-preserving embedding $[0,1]^{\sqcup n} \to  \partial \Sigma$ whose image we denote by $P$. 
This means, in the language of \cite{envas}, that $\Sigma$ is an operation of the \emph{open surface modular operad}. 

By what was explained in the preceeding section $\int_\Sigma \cA$ comes with the structure of a left $\int_{P \times [0,1]} \cA$-module category, and the orientation-preserving diffeomorphism $[0,1]^{\sqcup n}\ra{\cong} P$ makes $\int_\Sigma \cA$ into a module over $\cA^{\boxtimes n}$ with action
\begin{align}  \ogreaterthan:   \cA^{\boxtimes n} \boxtimes \int_\Sigma \cA \to \int_\Sigma \cA  \ , 
	\end{align} see also \cite[Section~5]{bzbj}.

The operadic composition of the open surface operad is the gluing along intervals.
To glue $\Sigma$ along a pair of intervals $I,J \subset \Sigma$, the parametrization of one of the two  intervals along which we glue (it does not matter whether we choose the first or the second interval of the pair) needs to precomposed with the reflection. 
  This turns the $\cA^{\boxtimes 2}$-action on $\int_{\Sigma}\cA$
   corresponding to $I$ and $J$  into an action of $\cA^{\env}$; we just have to precompose the $\cA$-action corresponding to the interval that is being reflected by the equivalence
   $\id: \cA \ra{\simeq} \cA^\rev$ whose structure maps are given by the braiding.
   This also gives us an orientation-reversing diffeomorphism $\psi:I\ra{\cong} J$.
   The result of gluing $\Sigma$ along $I$ and $J$ is the surface $\Sigma'$ obtained from $\Sigma$ by diving out by $x\sim \psi(x)$ for $x\in I$. 
Excision gives us an equivalence
\begin{align}\label{eq:equiv_excision_one_handle}
	\cA	\boxtimes_{\cA^\env} \int_\Sigma \cA \ra{\simeq}  
\int_{\Sigma'} \cA
 \ . 
\end{align}

\section{Pivotal module structure on the factorization homology of a unimodular finite ribbon category}

The goal of this section is to establish the pivotal structure 
on the factorization homology of a unimodular finite ribbon category. This will be accomplished in  \cref{thmmain} below.

\subsection{The special case of disks}\label{subsec:case_disks}

For a disk $D^2$ with $n$ marked boundary intervals, we have $\int_{D^2} \cA\simeq \cA$, and the corresponding action functor 
	$$ \ogreaterthan : \cA^{\boxtimes n} \boxtimes \cA \simeq \cA^{\boxtimes n+1} \to \cA  $$
	is naturally isomorphic, by property~\ref{propertyproduct} of \cref{subsec:review_FH}, to the $(n+1)$-fold monoidal product of $\cA$. Therefore, $\int_{D^2} \cA$ agrees with the left regular $\cA^{\boxtimes n}$-module $\cA$.

 Once we reverse the orientation of one or several of the marked boundary intervals to perform a gluing, we obtain the following module categories:

\begin{definition}\label{defApq}
If $\cA$ is a braided finite tensor category and integers
$p,q \geq 0$ with $p+q\ge 1$, we let  ${}_p \cA_q$ be the finite left $\cA^{\boxtimes p}\boxtimes \left( \cA^{\rev} \right)^{\boxtimes q}$-module category whose underlying finite category is $\cA$ with module structure determined by $$(X_1 \boxtimes \cdots \boxtimes X_p \boxtimes Y_1 \boxtimes \cdots \boxtimes Y_q) \ogreaterthan Z=  X_1 \otimes \cdots \otimes  X_p \otimes Z \otimes Y_q \otimes \cdots \otimes Y_1 \ .  $$
\end{definition}

\begin{remark}\label{remexact}
It is easy to see that ${}_p \cA_q$ is actually exact. Indeed by \cref{cor:projs_in_balanced_Deligne0}  the projective objects of the Deligne tensor product $\cA^{\boxtimes p}\boxtimes \left( \cA^{\rev} \right)^{\boxtimes q}$ are the direct summands of the Deligne products of projective objects. Now if $T$ is a direct summand of $P_1 \boxtimes \cdots \boxtimes P_{p+q}$, then $T \ogreaterthan X$ is a direct summand of $P_1 \boxtimes \cdots P_p \boxtimes X  \boxtimes P_{p+1} \boxtimes \cdots \boxtimes P_{p+q}$, which is projective, hence so is   $T \ogreaterthan X$.
\end{remark}

As a first step, we will prove that that the module categories ${}_p \cA_q$ are pivotal if $\cA$ is unimodular ribbon. To this end, we will use the following Lemma that follows from the standard fact that for any monoidal functor $\varphi : \mathcal{C} \to \mathcal{D}$ between finite tensor categories 
the restriction functor from $\cat{D}$-module categories to $\cat{C}$-module categories is a 2-functor:

\begin{lemma}\label{lemmapullback}
	Let $\mathcal{C}, \mathcal{D}$ be finite tensor categories and let $\varphi : \mathcal{C} \to \mathcal{D}$ be a monoidal functor. Suppose that  $\cM$ is a finite left $\mathcal{D}$-module category, $F: \cM \to \cM$ a $\mathcal{D}$-module endofunctor and $\chi: F \cong \id_\cM$ a trivialization of $F$ as $\mathcal{D}$-module functor. If $\mathcal{M}_\varphi$ denotes the left $\mathcal{C}$-module category induced by restriction of scalars, and $\varphi^* F:\mathcal{M}_\varphi  \to \mathcal{M}_\varphi $ is the left $\mathcal{C}$-module functor induced by $F$, then the natural isomorphism $\varphi^* \chi : \varphi^* F \cong \id_{\cM_\varphi}$ induced by $\chi$ is a trivialization of  $\varphi^* F$ as $\cC$-module functor.
\end{lemma}


\begin{proposition}\label{lem:disk_FH}
	Let $\cA$ be  a unimodular finite ribbon category with a  fixed trivialization $\alpha \cong \mathbbm{1}$ of the distinguished invertible object of $\cA$. Then ${}_p \cA_q$ inherits a structure of pivotal module category.
	In particular, for a disk $D^2$ with $n$ marked boundary intervals,  the trivialization $\alpha \cong \mathbbm{1}$ induces a pivotal $\cA^{\boxtimes n}$-module structure on $\int_{D^2} \cA$. 
\end{proposition}

\begin{proof}
We start by recalling that ${}_p \cA_q$ is exact  by \cref{remexact}. 
	For $(p,q)=(1,0)$ and $(p,q)=(0,1)$, the statement follows directly from \cref{exregularmodule}. 
	
	For $p=q=1$, we have using \eqref{eq:iso_2.15}  that
	 $\HOM(X,Y) \cong (Y \boxtimes \mathbbm{1}) \otimes \Delta \otimes (X^\vee \boxtimes \mathbbm{1})$, so using the self-duality of $\Delta$ given in \eqref{eqnpsi}  we have natural isomorphisms $$ \HOM (X, \sS^r(Y)) \cong \HOM(Y,X)^\vee \cong (Y^{\vee \vee} \boxtimes \mathbbm{1}) \otimes \Delta^\vee  \otimes (X^\vee \boxtimes \mathbbm{1}) \cong \HOM(X, Y^{\vee \vee})\,    $$
	and hence an isomorphism $\gamma: \sS^r \cong (-)^{\vee \vee}$.  To see that this is an isomorphism as $\cA^\env$-module functors one can argue as follows: If $W \in \cA^\env$, the commutativity of the diagram
$$
\begin{tikzcd}
\sS^r (W \ogreaterthan X) \rar \dar{\gamma_{W \ogreaterthan X}} & W^{\vee \vee} \ogreaterthan \sS^r(X) \dar{ \id_{W^
{\vee \vee}} \ogreaterthan \gamma_X }   \\
(W \ogreaterthan X)^{\vee\vee } \rar  &  W^
{\vee \vee} \ogreaterthan X^
{\vee \vee}
\end{tikzcd}
$$
amounts, by the Yoneda lemma, to the commutativity of the corresponding diagram obtained by applying $\HOM(Y,-)$. The latter can be seen to commute thanks to the compatibility between the left and right module structure of the internal hom as in \cite[Lemma 2.10]{schaumannpiv}.

	For general $p,q\ge 1$ we argue as follows: The distinguished invertible object of $\cA^{\boxtimes p}\boxtimes \left( \cA^{\rev} \right)^{\boxtimes q}$ is according to  \eqref{eq:disting_inv_for_Deligne_product} $ \alpha_{\cA^{\boxtimes p}\boxtimes \left( \cA^{\rev} \right)^{\boxtimes q}} \cong \alpha^{\boxtimes p+q} ,$ 
	which implies that $\cA^{\boxtimes p}\boxtimes \left( \cA^{\rev} \right)^{\boxtimes q}$ is unimodular as well with a trivialization induced by that of $\cA$. Moreover the pivotal structure of $\cA$ induces a canonical pivotal structure on $\cA^{\boxtimes p}\boxtimes \left( \cA^{\rev} \right)^{\boxtimes q}$. Therefore by \cref{lemmapivstr} it  suffices to obtain a trivialization of the right Nakayama functor of ${}_p \cA_q$ as $\cA^{\boxtimes p}\boxtimes \left( \cA^{\rev} \right)^{\boxtimes q}$-module functor. This follows directly from Lemma~\ref{lemmapullback} above applied to the functor $\varphi= \otimes^p \boxtimes \otimes^q : \cA^{\boxtimes p}\boxtimes \left( \cA^{\rev} \right)^{\boxtimes q} \to \cA \boxtimes \cA^\rev = \cA^\env$  (which is monoidal as $\cA$ is braided, see~\cite[Section~3.3]{bzbj}), $\cM={}_1 \cA_1$,  $F= \N^r$ the right Nakayama functor of $\cA$ and $\chi$ the trivialization $\N^r \cong \sS^r \cong \id_\cA$ resulting from $\alpha \cong \mathbbm{1}$ and the pivotal module structure for ${}_1 \cA_1$.
\end{proof}

\subsection{Gluing pivotal structures}

Any surface with boundary arises from gluing bands to a certain disjoint union of disks or, equivalently, from gluing disks along marked boundary intervals. The next step towards proving \cref{thmmain} is to show that the pivotal module structure in the factorization homology of  a marked disk discussed in the preceding section induces another such in the factorization homology of the gluing via the excision equivalence \eqref{eq:equiv_excision_one_handle}.

The core of the argument is based on a careful analysis of how a pivotal module structure in a certain bimodule category $\cM$ descends to the balanced Deligne product, which is of independent interest.

\begin{lemma}\label{lemma:projs_in_AxAeM}
	Let $\cA$ be a finite tensor category and let $\cM$ be a finite $\cA$-bimodule category.
	The projective objects of the balanced Deligne product $\cA \boxtimes_{\cA^\env} \cM$ are direct summands of objects of the form $\mathbbm{1} \boxtimes_{\cA^\env} P$ for $P \in \cM$ projective.\label{lemmacy1}
\end{lemma}

\begin{proof}
	By Corollary~\ref{cor:projs_in_balanced_Deligne0} the projective objects 
	of
	$\cA \boxtimes_{\cA^\env} \cM$ are direct summands of objects of the form
	$X \boxtimes P$ with $X\in\cA$ and $P\in\Proj \cat{M}$.
	But in $\cA \boxtimes_{\cA^\env} \cM$ such an object is isomorphic to $\mathbbm{1}\boxtimes (X\ogreaterthan P)$, where $X\ogreaterthan P$ is again projective. 
\end{proof}

If $\cA$ and $\cat{B}$
	are  finite tensor categories  
	 and  $\cM$ is a finite left $\cA^\env \boxtimes \cB$-module category,  note that $\cM$ is in particular a left $\cA^\env$- and $\cB$-module category with
\begin{equation*}
X \ogreaterthan M := (X \boxtimes \mathbbm{1}_\cB)\ogreaterthan M \qquad , \qquad Y \ogreaterthan M:=(\mathbbm{1}_{\cA^\env} \boxtimes Y) \ogreaterthan M 
\end{equation*}
for $X \in \cA^\env$, $Y \in \cB$ and $M  \in M$. Observe that in this case, the balanced Deligne product $\cA \boxtimes_{\cA^\env} \cM$ inherits the structure of finite left $\cB$-module category via $B \ogreaterthan (A \boxtimes_{\cA^\env} M) := A \boxtimes_{\cA^\env} (B \ogreaterthan M)$.

\begin{proposition}\label{propgluing}
	Let $\cA$ and $\cat{B}$
	be   pivotal finite tensor categories,  
	 and let $\cM$ be a finite left $\cA^\env\boxtimes \cB$-module category.
\begin{pnum}
\item There is a natural isomorphism 
$$
\Hom_{\cA	\boxtimes_{\cA^\env} \cM}( \mathbbm{1} \boxtimes_{\cA^\env} M,\mathbbm{1} \boxtimes_{\cA^\env} M'  ) \cong \Hom_{\cM} (    M, \Delta    \ogreaterthan   M') \ . \label{gluing1.1}
$$ 
	 \item If $\HOM_{ \cM}$ denotes the internal hom with respect to the $\cB$-module structure of $\cM$, then there is a natural isomorphism
$$
\HOM_{\cA	\boxtimes_{\cA^\env} \cM}( \mathbbm{1} \boxtimes_{\cA^\env} M,\mathbbm{1} \boxtimes_{\cA^\env} M'  )  \cong \HOM_{\cM} (    M, \Delta    \ogreaterthan   M') \ . \label{gluing1.2}
$$
 \item The right Nakayama functor $\nakar$ of $\cA \boxtimes_{\cA^\env} \cM$ is determined, up to natural isomorphism, by the restriction of the right Nakayama functor  of $\cM$ to $\Proj \cM$. More precisely, we have natural isomorphisms 
$$
 	\nakar (\mathbbm{1}\boxtimes_{\cA^\env} P)\cong \mathbbm{1}\boxtimes_{\cA^\env}  (\alpha \boxtimes \mathbbm{1}) \ogreaterthan \nakar_{\cM} P 
$$
 for all projective objects $P\in \cM$, where $\alpha \in \cA$ denotes the distinguished invertible object . \label{gluing2}
 \end{pnum}
	\end{proposition}

\begin{proof}
For \ref{gluing1.1} we first we observe with the definition $\Delta :=\END(\mathbbm{1}) \in \cA^\env$ that
	\begin{align*}\Hom_{\cA	\boxtimes_{\cA^\env} \cM}( \mathbbm{1} \boxtimes_{\cA^\env} M,\mathbbm{1} \boxtimes_{\cA^\env} M'  ) &\cong 
		\Hom_{\cA^\env} (    \mathbbm{1}\boxtimes\mathbbm{1} ,\Delta \otimes  \HOM(M, M') )\\
&\cong \Hom_{\cA^\env} (    \mathbbm{1}\boxtimes\mathbbm{1} , \HOM(M,\Delta \ogreaterthan M') )\\
&\cong \Hom_{\cM} (   M , \Delta \ogreaterthan M')
		 \ , 
\end{align*}
where in the first isomorphism we have used \eqref{eq:iso_DSS}, the second one is induced by \eqref{eq:iso_2.15} and the third one is \eqref{eq:internal_hom_adj}.

Next for \ref{gluing1.2}, we have for every $B\in\cB$ that
\begin{align*}
	\Hom_{\cB}(B,\HOM_ {\cA	\boxtimes_{\cA^\env} \cM}(\mathbbm{1}\boxtimes_{\cA^\env} M , \mathbbm{1}\boxtimes_{\cA^\env} M')   )&\cong \Hom_{\cA	\boxtimes_{\cA^\env} \cM}
	(\mathbbm{1}\boxtimes_{\cA^\env} B\ogreaterthan M,\mathbbm{1}\boxtimes_{\cA^\env} M')\\ &\cong \Hom_{\cM}(B \ogreaterthan M,\Delta \ogreaterthan M') \\&\cong \Hom_{\cB}(B,\HOM_{\cM} (    M, \Delta    \ogreaterthan   M')) \ ,
	\end{align*}
where the first isomorphism is \eqref{eq:internal_hom_adj}, the second is \ref{gluing1.1} and the third is again \eqref{eq:internal_hom_adj}.

Lastly for \ref{gluing2}, we have the following chain of isomorphisms, that we explain below:
	\begin{align*} \Hom_{\cA	\boxtimes_{\cA^\env} \cM}( \mathbbm{1} \boxtimes_{\cA^\env} P,\mathbbm{1} \boxtimes_{\cA^\env} Q  ) &\cong \Hom_{\cM} (P, \Delta \ogreaterthan Q)\\
	&\cong  \Hom_{\cM} (\Delta \ogreaterthan Q, \nakar_{\cM}P)^*\\
		&\cong  \Hom_{\cM} ( Q, \Delta^\vee \ogreaterthan\nakar_{\cM}P)^* \\
&\cong  \Hom_{\cM} ( Q, \Delta \ogreaterthan (\alpha \boxtimes \mathbbm{1}) \ogreaterthan\nakar_{\cM}P)^* \\
&\cong \Hom_{\cA	\boxtimes_{\cA^\env} \cM}( \mathbbm{1} \boxtimes_{\cA^\env} Q,\mathbbm{1} \boxtimes_{\cA^\env}  (\alpha \boxtimes \mathbbm{1}) \ogreaterthan \nakar_{\cM} P  )^* 
\end{align*} 
The first isomorphism is \ref{gluing1.1}, the second is \eqref{eq:Nr-twisted_CY}, the third is \eqref{eq:adjunctions_action_functor} (here we identify ${}^\vee \Delta \cong \Delta^\vee$ thanks to the pivotal structure), the fourth is induced by $\Delta^\vee \cong (\alpha \boxtimes \mathbbm{1}) \otimes \Delta\cong \Delta \otimes (\alpha \boxtimes \mathbbm{1})$ via the half braiding of $\alpha$ from \cite[Lemma~2.1]{mwcenter}, and the fifth is again \ref{gluing1.1}. We conclude by \eqref{eq:Nr-twisted_CY} and the Yoneda lemma (the fact that the Nakayama functor
of	 $\cA \boxtimes_{\cA^\env} \cM$
	 is determined by  this follows from Lemma~\ref{lemma:projs_in_AxAeM}).
	\end{proof}

\begin{proposition}\label{propgluing2}
	Let $\cA$ and $\cat{B}$
	be unimodular  pivotal finite tensor categories
	with fixed trivializations of their respective distinguished invertible objects,  
	and let $\cM$ be an exact left $\cA^\env\boxtimes \cB$-module category such that $\cM$ is pivotal  viewed  as a $\cB$-module category.
	Then
\begin{pnum}
\item The right Nakayama functor of  $\cA \boxtimes_{\cA^\env} \cM$ is trivializable.\label{propgluing2(i)}
\item  $\cA \boxtimes_{\cA^\env} \cM$ is an exact $\cB$-module category.\label{propgluing2(ii)}
\end{pnum}
	\end{proposition}
\begin{proof}
By \cref{lemmapivstr}, the right  Nakayama functor of $\cM$ comes with a ($\cB$-module) isomorphism
	$ \nakar_{\cM}\cong \id_{\cM}$. This trivialization and the fixed trivialization $\alpha \cong \mathbbm{1}$ of the distinguished invertible object of $\cA$ induce, by  \cref{propgluing}~\ref{gluing2},  a trivialization $\chi$ of the Nakayama functor of $\cA \boxtimes_{\cA^\env} \cM$. This proves \ref{propgluing2(i)}.

Let $A\in \cA^\env$, $B \in\cB$ and $M,M'\in\cM$.	Then for the $\cB$-valued internal homs $\HOM_\cM$ of $\cM$ we find
	\begin{align*}
		\Hom_{\cB}(B,\HOM_{\cM}(A\ogreaterthan M,M'))&\cong
		\Hom_{\cM} ((A \boxtimes B)\ogreaterthan M,M') \\ &\cong \Hom_{\cM}(      B\ogreaterthan M,A^\vee \ogreaterthan M') \\&\cong \Hom_{\cB}(B,\HOM_{\cM} (M,A^\vee \ogreaterthan M')) \ , 
		\end{align*}
	and therefore 
	\begin{equation}\label{eq:iso_Sch_bimod}
		\HOM_{\cM}(A\ogreaterthan M,M')\cong 
		\HOM_{\cM} (M,A^\vee \ogreaterthan M') \ , 
		\end{equation}
see also \cite[Lemma 2.18]{schaumannpiv}.	 
	Then we obtain the following chain of natural isomorphisms:
\begin{align*}
\HOM_{\cA	\boxtimes_{\cA^\env} \cM}( \mathbbm{1} \boxtimes_{\cA^\env} M,\mathbbm{1} \boxtimes_{\cA^\env} M'  )^\vee &\cong \HOM_{\cM} (    M, \Delta    \ogreaterthan   M')^\vee\\
&\cong \HOM_{\cM} (  \Delta    \ogreaterthan   M,  M'  ) \\
&\cong \HOM_{\cM} (    M', \Delta    \ogreaterthan   M) \\
&\cong \HOM_{\cA	\boxtimes_{\cA^\env} \cM}( \mathbbm{1} \boxtimes_{\cA^\env} M',\mathbbm{1} \boxtimes_{\cA^\env} M  ) \ .
\end{align*}
Here the first and last isomorphisms are \cref{propgluing}~\ref{gluing1.2}, the second isomorphism is induced by the pivotal module structure of $\cM$, and the third one is induced by combining \eqref{eq:iso_Sch_bimod} and the self-duality of $\Delta$ from \eqref{eqnpsi} 
because $\cA$ is unimodular. Since all objects in $\cA	\boxtimes_{\cA^\env} \cM$ 
are of the form $\mathbbm{1} \boxtimes_{\cA^\env} M$, this tells us that the internal hom is exact, as it is left exact and therefore its dual is right exact. Hence, all objects of    $\cA	\boxtimes_{\cA^\env} \cM$ are $\cB$-projective, that is,  $\cA	\boxtimes_{\cA^\env} \cM$ is an exact $\cB$-module category, which shows \ref{propgluing2(ii)} (or one argues via \cite[Proposition~7.9.7~(1)]{egno}).
	\end{proof}

Together with \eqref{eq:equiv_excision_one_handle}, this immediately implies

\begin{corollary}
Let $\cat{A}$ be a unimodular finite ribbon category,	and let $\Sigma$ be
	a compact, oriented surface  with $n$ marked boundary intervals, at least one per connected component.
	Then the factorization homology $\int_\Sigma \cat{A}$
	is an exact left
	 $\cat{A}^{\boxtimes n}$-module category.
\end{corollary}

\begin{example}\label{exaenv}
	In Proposition~\ref{propgluing2} set $\cM=\cA^\env$ with the $ \cA^\env$-action being the regular one (the assumption on $\cA$ and $\cB$ remain in place). This tells us that $\cA$ and $\cA \boxtimes_{\cA^\env} \cA^\env$ comes with a trivialization of the respective right Nakayama functors. By explicitly calculating the latter we will show that these trivializations correspond to each other under the canonical equivalence
	\begin{equation} \label{eqnequiv}
	\cA \boxtimes_{\cA^\env} \cA^\env \ra{\simeq} \cA \qquad , \qquad A \boxtimes_{\cA^\env} X \mapsto X \ogreaterthan A \ . 
		\end{equation} Indeed, \eqref{eq:N=alpha.(-)**} and \eqref{eq:disting_inv_for_Deligne_product} give us $\nakar_{\cA^\env}\cong (\alpha^{-1}\boxtimes \alpha^{-1}) \otimes (-)^{\vee \vee}$
	and therefore with \cref{propgluing}~\ref{gluing2}
	\begin{equation} \label{eq:N(1xP)=1x...}
	 \nakar (\mathbbm{1}\boxtimes_{\cA^\env} P)\cong \mathbbm{1}\boxtimes_{\cA^\env}  ( \mathbbm{1}\boxtimes \alpha ^{-1}) \otimes  P^{\vee \vee}
\end{equation}
 for $P\in \Proj \cA^\env$. 
  Since~\eqref{eqnequiv} is induced by the monoidal product, this means that the previous isomorphism \eqref{eq:N(1xP)=1x...} corresponds, under the canonical equivalence, to the natural isomorphism \eqref{eq:N=alpha.(-)**}. The statement follows since the trivialization of the latter as well as the trivialization of 	\eqref{eq:N(1xP)=1x...} are both induced by the fixed isomorphism $\alpha \cong \mathbbm{1}$ and the pivotal structure of $\cA$.
	\end{example}

\subsection{Trace functions\label{sectraces}}
Recall from \eqref{eq:Nr-twisted_CY} that any finite category comes with a \emph{right $\nakar$-twisted Calabi-Yau structure} in the sense that  there are natural isomorphisms
\begin{equation*}
\Hom_{\cM} (P,Q) \cong \Hom_{\cM}(Q, \N^r P)^*  \ ,
\end{equation*}
so a trivialization of $\N^r$ as an ordinary functor gives us 
 non-degenerate traces
$$ \tr^{\!\cM}_P : \Hom_{\cM}(P,P)\to k$$
that are cyclic and hence
descend to $H \! H_0(\cM) := \int^{X\in \Proj \cM} \mathsf{End}_\cM(X)$  and assemble into a map
$$
\tr^{\!\cM} : H\! H_0(\cat{M})\to k \ . 
$$ For $P\in \Proj\cM$ and $g:P\to P$, we denote the trace  by
$$
\tr^{\!\cM} \left(	\begin{array}{c}\begin{tikzpicture}
	\begin{pgfonlayer}{nodelayer}
		\node [style=none] (1) at (-0.5, 1) {};
		\node [style=none] (2) at (0, 1) {};
		\node [style=bigbox] (3) at (0, 1) {$g$};
		\node [style=none] (5) at (0, 2.5) {};
		\node [style=none] (7) at (0, -0.5) {};
		\node [style=none] (8) at (0.5, -0.5) {$P$};
		\node [style=none] (9) at (0.5, 2.5) {$P$};
	\end{pgfonlayer}
	\begin{pgfonlayer}{edgelayer}
		\draw (2.center) to (5.center);
		\draw (3) to (7.center);
	\end{pgfonlayer}
\end{tikzpicture}
\end{array}\right)
$$
 in the graphical calculus to be read from bottom to top.

Now suppose that $\cM$ is actually a module over a pivotal finite tensor category $\cA$.
By \cite[Section~5.1]{shibatashimizu}, the traces corresponding to trivializations of $\N^r$ as $\cA$-module functor  are exactly
the ones that satisfy the \emph{partial trace property}
\begin{equation}
\tr(f)=\tr (\close_{X|P}(f)) \ , \label{eqnptp}
\end{equation}
for all $X\in\cA$, $P\in\Proj \cM$ and $f: X \ogreaterthan P \to  X \ogreaterthan P$ 
for the  closing operator
$$ \close_{X|P}:\Hom_{\cM}(X \ogreaterthan P,X \ogreaterthan P )\cong \Hom_{\cA}           (P,(X^\vee \otimes X )\ogreaterthan P) \ra{ b_X  }   
\Hom_{\cM}( P, P ) \  
$$

\begin{proposition}
Let $\cA$ and $\cat{B}$
	be unimodular  pivotal finite tensor categories
	with fixed trivializations of their respective distinguished invertible objects,  
	and let $\cM$ be an exact left $\cA^\env\boxtimes \cB$-module category such that $\cM$ is pivotal  viewed  as a $\cB$-module category. Then the trivialization of the right Nakayama functor of $\cA \boxtimes_{\cA^\env} \cM$ given in   \cref{propgluing2}~\ref{propgluing2(i)} is in fact a trivialization as $\cB$-module functor, that is,  $\cA \boxtimes_{\cA^\env} \cM$ inherits a pivotal module structure from that of $\cM$.
\end{proposition}
\begin{proof}
We start by noting that $\cA \boxtimes_{\cA^\env} \cM$ is exact by  \cref{propgluing2}~\ref{propgluing2(ii)}. By the discussion above, it suffices to show that the traces associated to the trivialization of \cref{propgluing2}~\ref{propgluing2(i)}
satisfy the partial trace property. By \cref{lemma:projs_in_AxAeM}, it is enough to verify this for endomorphisms
	 $f: \mathbbm{1} \boxtimes_{\cA^\env} P \to \mathbbm{1} \boxtimes_{\cA^\env} P$ with $P\in \Proj \cM$.
	 
	From the construction in \cref{propgluing} and the correspondence between traces and trivializations of the Nakayama functor in \cite{shibatashimizu,tracesw}, one obtains that
	\begin{align}
	\tr^{\!\cA \boxtimes_{\cA^\env} \cM }(f)=	\tr^{\!\cM}  \left(	\begin{array}{c}	\begin{tikzpicture}
	\begin{pgfonlayer}{nodelayer}
		\node [style=none] (1) at (-0.5, 1) {};
		\node [style=none] (2) at (0.5, 1) {};
		\node [style=bigbox] (3) at (0, 1) {$\widetilde f$};
		\node [style=none] (4) at (-0.5, 4) {};
		\node [style=none] (5) at (0.5, 6) {};
		\node [style=bigbox] (6) at (-0.5, 4) {$\varepsilon$};
		\node [style=none] (7) at (0, -2) {};
		\node [style=none] (8) at (0.5, -2) {$P$};
		\node [style=none] (9) at (1, 6) {$P$};
		\node [style=none] (10) at (-1.25, 2.25) {$\Delta$};
		\node [style=none] (11) at (0, 2.25) {$\ogreaterthan$};
	\end{pgfonlayer}
	\begin{pgfonlayer}{edgelayer}
		\draw (1.center) to (4.center);
		\draw (2.center) to (5.center);
		\draw (3) to (7.center);
	\end{pgfonlayer}
\end{tikzpicture}
\end{array}\right) \ , 
	\end{align}
	where 
	\begin{itemize}
		\item $\widetilde f : P \to \Delta \ogreaterthan P$ is the image of $f$ under the isomorphism of \cref{propgluing}~\ref{gluing1.1},
		\item $\varepsilon : \Delta \to \mathbbm{1}\boxtimes\mathbbm{1}$ is the Frobenius form of $\Delta$ (which is a symmetric Frobenius algebra by \cite[Theorem 3.14]{relserre}).
	\end{itemize}
	In this description, it is clear that this trace has the partial trace property~\eqref{eqnptp}
	with respect to the remaining $\cB$-action, because it is by assumption already the case for $	\tr^{\!\cM}  $, and the application of $\varepsilon$ does not change that.
This finishes the proof.
	\end{proof}

\subsection{The main result}

Once we have proven how a pivotal module structure  induces another such on (a special kind of) the balanced Deligne product, we are in the position to prove our main result on the pivotal module structure of factorization homology.

\begin{theorem}\label{thmmain}
Let $\cat{A}$ be a unimodular finite ribbon category with chosen trivialization $\alpha \cong \mathbbm{1}$ of the distinguished invertible object,
	and let $\Sigma$ be
	a compact, oriented surface  with $n$ marked boundary intervals, at least one per connected component.
	Then the factorization homology $\int_\Sigma \cat{A}$
	comes with the structure of a pivotal left
	 $\cat{A}^{\boxtimes n}$-module category. 
	\end{theorem}

\begin{proof}
	Without loss of generality, we may assume that $\Sigma$ is connected. 
	If $\Sigma$ is a disk, the statement is true by	 \cref{lem:disk_FH}. In the general case, we choose a ribbon graph model for the surface, see e.g.~\cite{egaskupers} for the necessary background. 
A ribbon graph $R$ for $\Sigma$ is a finite graph such that the graph obtained by contracting all internal edges is a corolla whose $n$ legs are identified with the boundary intervals of $\Sigma$, with the datum of a cyclic orders for each of the vertices of $R$ and an identification $|R|\ra{\cong} \Sigma$ between the surface obtained by geometrically realizing $R$ and $\Sigma$.
More explicitly, the $m$ vertices of $R$ with their cyclic order are fattened into disks that are then glued along intervals as prescribed by $R$. 

The ribbon graph $R$ in combination with the excision property for factorization homology offers a way to obtain $\int_\Sigma \cA$ by taking a number of the modules ${}_p \cA_q$ from  \cref{defApq}, namely one for each of the vertices of $R$, and taking a relative tensor product $\boxtimes_{\cA^\env}$, namely for each gluing along an interval. This gives us an $\cat{A}^{\boxtimes n}$-module category $\cC(\Sigma;R)$ plus an equivalence $\varphi_R:\cC(\Sigma;R)\ra{\simeq}\int_\Sigma \cA$ of $\cA^{\boxtimes n}$-modules. By repeatedly applying \cref{propgluing2}, we obtain a pivotal $\cA^{\boxtimes n}$-module structure on $\cC(\Sigma;R)$.
This then makes $\int_\Sigma \cA$ also a pivotal $\cA^{\boxtimes n}$-module category via $\varphi_R$.

It remains to prove that this does not depend on the ribbon graph model $R$.
Suppose we have another ribbon graph model $R'$ for $\Sigma$,
then $R'$ arises from $R$ through a zigzag of tree contractions. Without loss of generality, we can assume that $R'$ arises from $R$ by the contraction of an edge attached to two different vertices. This gives us an equivalence $\gamma: \cC(\Sigma;R)\ra{\simeq} \cC(\Sigma;R')$, namely the one eliminating a relative tensor product $\boxtimes_{\cA^\env} \cA^{\env}$, that satisfies $\varphi_{R'}\circ\gamma=\varphi_R$.  It suffices to see that $\gamma$ preserves the trivialization of the right Nakayama functor on $\cC(\Sigma;R)$ and $\cC(\Sigma;R')$
that gives us the pivotal structure, but this was already observed in \cref{exaenv}.
\end{proof}

\begin{remark}
	If $\cA$ is modular and $\Sigma$ has exactly one interval per boundary component, then any handlebody $H$ for $\Sigma$ gives us an equivalence
	$\int_\Sigma \cA\simeq \cA^{\boxtimes n}$ as $\cA^{\boxtimes n}$-module categories, where the $\cA^{\boxtimes n}$-action on $\cA^{\boxtimes n}$ is the regular one~\cite[Theorem 6.5]{reflection}.
	In that case, \cref{thmmain} is clear.
	\end{remark}

\begin{corollary}[Modular invariance]\label{corinvariance}
Under the hypotheses of \cref{thmmain}, the group of diffeomorphisms of $\Sigma$ preserving the orientation and the parametrized intervals acts on $\int_{\Sigma}\cA$ by $\cA^{\boxtimes n}$-module automorphisms preserving the pivotal module  structure.
\end{corollary}

\begin{proof}
	If $f:\Sigma \to \Sigma$ is an diffeomorphism preserving the orientation and the boundary parametrization, the induced map $f_* : \int_\Sigma \cA \ra{\simeq} \int_{\Sigma}\cA$ comes canonically with an $\cA^{\boxtimes n}$-module structure, see e.g.~the comments in \cite[Remark 4.5]{brochierwoike}. Moreover, $f_*$ sends the pivotal structure on $\int_\Sigma \cA$ inherited from some ribbon graph model $\cC(\Sigma;R)\ra{\simeq} \int_\Sigma \cA$ to the one inherited from the ribbon graph model for $\Sigma$ that has the same ribbon graph but whose identification $|R|\ra{\cong} \Sigma$ is postcomposed with $f$. This however gives us, by the independence of the ribbon graph model observed in the proof of \cref{thmmain}, the \emph{same} pivotal module structure. Therefore, $f_*$ is also compatible with the pivotal module structure. 
	\end{proof}

\section{Applications}

In this last section,  we explain some applications of our main result: We obtain symmetric Frobenius algebras from factorization homology and we use them to construct consistent systems of correlators for open conformal field theories. Moreover, we construct modular invariant modified traces.
All of these applications have already been motivated in the introduction, so we allow ourselves to keep this section short and just state the results.

\subsection{Frobenius structure on moduli algebras}
If $\mathcal{A}$ is a pivotal finite tensor category and $F=(F, \mu, \eta)$ is an algebra object in $\cat{A}$, a \emph{Frobenius structure} on $F$ is a morphism $\lambda : F \to \mathbbm{1}$ in $\cat{C}$, called the \emph{Frobenius form}, such that the composite 
$$ \psi : F \ra{\id_F \otimes b_F \ }  F \otimes F \otimes F^\vee \ra{\lambda\circ \mu \otimes \id_{F^\vee}}  F^\vee
	 $$
is an isomorphism. We say additionally that the Frobenius algebra $(F, \lambda)$ is \emph{symmetric} if the composite 
$F \ra{  \omega_F    } F^{\vee \vee}  \ra{   \psi^\vee    }  F^\vee$ agrees with $\psi$, where $\omega: \id \cong (-)^{\vee \vee}$ denotes the pivotal structure of $\cat{A}$. We refer the reader to \cite{fuchsstigner} for more background.

																																							As a consequence of our main result, we will now give a symmetric Frobenius structure on the moduli algebra
$\mathfrak{a}_\Sigma:=\END(\mathcal{O}_\Sigma^\cA) \in \cA^{\boxtimes n}$ for a unimodular finite ribbon category $\cA$, with the assumptions on $\Sigma$ being the ones from \cref{thmmain}.	By 	\cite[Section 5.4]{bzbj},	 	the moduli algebra 	$\mathfrak{a}_\Sigma$	 carries a mapping class group action which is induced by the homotopy fixed point structure of $\mathcal{O}_\Sigma^\cA$ with respect to the canonical diffeomorphism group action on factorization homology.																											

To further illustrate the significance of this algebra, let $n=1$ and $\Sigma$ be connected.
Then, as an object of $\cA$, we have $\mathfrak{a}_\Sigma \cong \mathbb{F}^{\otimes 2g+r-1}$ where $\mathbb{F} =
\int^{X \in \cA} X^\vee  \otimes X  $ if $\Sigma$ has genus $g$ and $r$ boundary components.   The relevance of this object comes from \cite[Theorem 5.14]{bzbj}, where it is shown that there is an equivalence of left $\cA$-module categories
\begin{equation}\label{eq:FH=mod_MA}
	\int_\Sigma \cA \simeq \mathsf{mod}_\cA\text{-}\mathfrak{a}_\Sigma .
\end{equation}

\begin{corollary}\label{cor:a_Sigma_sym_Frob}
	Let $\cat{A}$ be a unimodular finite ribbon category with chosen trivialization $\alpha \cong \mathbbm{1}$ of the distinguished invertible object,
	and let $\Sigma$ be
	a compact, oriented surface  with $n$ marked boundary intervals, at least one per connected component.
Then the moduli algebra $\mathfrak{a}_\Sigma\in\cA^{\boxtimes n}$ comes canonically 
with the structure of a symmetric Frobenius algebra, and the mapping class group action on $\mathfrak{a}_\Sigma$ 
preserves the Frobenius structure.
\end{corollary}

\begin{proof}
The algebra 	$\mathfrak{a}_\Sigma := \END(\cat{O}_\Sigma^{\cA})$
inherits a symmetric Frobenius structure
from the pivotal $\cA^{\boxtimes n}$-structure 
from \cref{thmmain} thanks to \cite[Theorem 3.15]{relserre}.
The mapping class group action preserves the Frobenius structure by \cref{corinvariance}.
\end{proof}

\cref{thmmain} also allows us to make a statement about the center of moduli algebras, for the notion of center suggested in \cite{internal}:

\begin{corollary}
	Let $\cat{A}$ be a unimodular finite ribbon category with chosen trivialization $\alpha \cong \mathbbm{1}$ of the distinguished invertible object,
	and let $\Sigma$ be
	a compact, oriented, connected surface  with one marked boundary interval.
	Then the center
	\begin{align}
		Z(\mathfrak{a}_\Sigma) = \int_{X \in \mathsf{mod}_\cA\text{-}\mathfrak{a}_\Sigma} \END(X) \in Z(\cA)
	\end{align}
	is a symmetric commutative Frobenius algebra in the Drinfeld center $Z(\cA)$.
\end{corollary}

\begin{proof}
	This follows at once from \cref{thmmain}, the equivalence \eqref{eq:FH=mod_MA} and \cite[Corollary~39]{internal}.
\end{proof}

\subsection{Open correlators}
As explained in the introduction, \cref{cor:a_Sigma_sym_Frob} has the following implication in terms of open correlators:

\begin{corollary}\label{coropen}Let $\cA$ be a unimodular finite ribbon category and denote by $\Ao$ the unique open modular functor whose evaluation on disks with marked boundary intervals is given by $\cA$ as pivotal finite tensor category.
	Fix a compact, connected, oriented surface $\Omega$ with a single marked boundary interval. 
	Then the moduli algebra $\mathfrak{a}_\Omega$ gives rise to a consistent system of open correlators
$$ \xi_\Sigma^{\Omega} \in \Ao(\Sigma; \mathfrak{a}_\Omega,\dots,\mathfrak{a}_\Omega) \quad \text{for all}\quad \Sigma \in \O(n)$$
	that define a full open conformal field theory $(\Ao,\xi^\Omega)$. 
	Explicitly, the datum of $ \xi_\Sigma^{\Omega}$ amounts equivalently to 
	\begin{pnum}
		\item a map
$$
		\cat{O}_\Sigma^\cat{A} \to \mathfrak{a}_\Omega^{\boxtimes n} \ogreaterthan\cat{O}_\Sigma^\cat{A}
$$
in $\int_\Sigma \cat{A}$
that is $\Map(\Sigma)$-invariant with respect to the  action of the mapping class group $\Map(\Sigma)$ which arises from the homotopy fixed point structure of $\cat{O}_\Sigma^\cat{A}$,\label{item:i_cor}

	\item a map 
$$
		\mathfrak{a}_\Omega^{\boxtimes n} \to \mathfrak{a}_\Sigma
$$ in $\cA^{\boxtimes n}$ that is invariant under the action of $\Map(\Sigma)$ on $\mathfrak{a}_\Sigma$ constructed by Alekseev-Grosse-Schomerus and more generally Ben-Zvi-Brochier-Jordan.  \label{item:ii_cor}
\end{pnum}
	\end{corollary}

\begin{proof}
	The main statement follows from \cref{cor:a_Sigma_sym_Frob}
	because of the classification of open correlators in \cite[Thereom~8.3]{microcosm}.
	Now the first point follows once we combine it with \cite[Theorem 6.1]{correlators}, and the second one is an immediate consequence of \ref{item:i_cor} using \eqref{eq:adjunctions_action_functor} and the self-duality of $\mathfrak{a}_\Omega$ provided by its Frobenius algebra structure. 
	\end{proof}

\begin{example}\label{excalculation}
In \cref{coropen}, let us assume that $\cA$ is given by the finite-dimensional modules over a finite-dimensional ribbon Hopf algebra $H$.
Then the coevaluation $k\to \mathfrak{a}_\Omega^* \otimes \mathfrak{a}_\Omega$ of $\mathfrak{a}_\Omega$ sends $1\in k$ 
to $\sum_i \alpha_i \otimes v_i$ with $\alpha_i\in \mathfrak{a}_\Omega^*$ and $v_i \in \mathfrak{a}_\Omega$. 
We concentrate now on the case in which $\Sigma$ is an annulus with one marked boundary interval.
In that case, $\mathfrak{a}_\Sigma$ is the 
dual of $H$ with the coadjoint action:
$\mathfrak{a}_\Sigma=H_\text{coadj}^*$ \cite{bzbj}, see also the comments before~\eqref{eq:FH=mod_MA}. 
If we denote the product on $\mathfrak{a}_\Omega$ (originating from the composition of internal endomorphisms) by $\circ$, then with the calculation recipe from \cite[Example 9.1]{microcosm}, we deduce that the open correlator for the annulus amounts to the map
$$
	\mathfrak{a}_\Omega \to H_\text{coadj}^* \ , \quad a \mapsto \sum_i \alpha_i ( -   \cdot (  v_i \circ a   )      ) \ , 
$$
	where $\cdot$ is the multiplication in $H$.
	In order to get the map for $\Sigma = \mathbb{T}^2_1$, the torus with one boundary component and one marked interval,
	one essentially needs to repeat this procedure for two copies of $H_\text{coadj}^*$. One then recovers the elliptic class function of the symmetric Frobenius algebra $\mathfrak{a}_\Omega$, see \cite[Section~7.6]{correlators} for details.
		\end{example}

\begin{remark}\label{remsymmetry}
	From \cite[Thereom~8.3 and Section~12]{microcosm}, it follows that the mapping class group of $\Omega$ acts by automorphisms of $(\Ao,\xi^\Omega)$. 
	\end{remark}	

\begin{remark} Denote by $\mathfrak{F}_{Z(\cA)}$ the Lyubashenko modular functor
	\cite{lyu,lyubacmp, lyulex} for the Drinfeld center $Z(\cat{A})$ of $\cat{A}$.
	Then \cite[Theorem 4.3]{envas} implies that the open correlators from \cref{coropen} amount to
mapping class group invariant vectors in $\mathfrak{F}_{Z(\cA)}(\Sigma;  L\mathfrak{a}_\Omega,\dots,L\mathfrak{a}_\Omega      )$, where $L:\cA\to Z(\cA)$ is the left adjoint to the forgetful functor $U:Z(\cA)\to\cA$. 
Thanks to the string-net model for $\mathfrak{F}_{Z(\cA)}$ \cite{sn}, these vectors 
can also be seen as mapping class group invariant string-nets.
\end{remark}

\subsection{Modular invariant modified traces\label{secmodtrace}}
By means of the modular invariance of the pivotal structure featuring in our main result (\cref{corinvariance}),
we may now construct modular invariant traces that can be seen as an integration of modified traces over surfaces (as motivated in the introduction):

\begin{theorem}[Modular invariant modified traces]\label{thmmodtrace}
	
	\end{theorem}

\begin{proof}
	The trivialization of the Nakayama functor of $\int_\Sigma \cA$ that gave us the pivotal $\cA^{\boxtimes n}$-module structure leads to the traces \eqref{eqntracefunctions}.
	The property (N) is clear, and (P) is simply a consequence of the fact 
	that this trivialization is a trivialization as module functor, invoking again \cite[Section~5.1]{shibatashimizu}.
	Property~(M) is a consequence of \cref{corinvariance}, and (L) holds by construction (because we glue the pivotal structures via \cref{propgluing2}).

	The restriction to disks is exactly a trivialization of $\naka$ as module functor relative to the pivotal structure, which by \cite{shibatashimizu,tracesw} is exactly a two-sided modified trace on $\Proj \cA$. Such a trivialization extends from disks to surfaces in a unique way respecting (N) and (L).
	\end{proof}

	\begin{remark}
		If $\cA$ is not only a unimodular finite ribbon category, but actually a modular category, then
		the mapping class group acts on $\int_\Sigma \cA$ by functors that, as $\cA^{\boxtimes n}$-module maps, are all isomorphic to the identity functor~\cite[Proposition 7.7]{brochierwoike} (but not canonically so --- in fact, the action amounts to a 2-cocycle on the mapping class group encoding the framing anomaly). In that situation, (M) holds for any cyclic trace, but if $\cA$ is not modular the mapping class group action on factorization homology is not trivial: In fact, if it were, $\cA$, as a cyclic framed $E_2$-algebra, would be connected and hence extends to a modular functor by the main result of \cite{brochierwoike}. But then \cite[Proposition 5.8]{reflection} would imply that $\cA$ is actually modular (because we have assumed that $\cA$ is finite ribbon and unimodular).
		\end{remark}

	\small	
	\bibliographystyle{alpha}
	\bibliography{ref}
	\vspace*{0.3cm} \noindent  \textsc{Université Bourgogne Europe, CNRS, IMB UMR 5584, F-21000 Dijon, France}
	
\end{document}